\documentclass[reqno,oneside,a4paper,11pt]{amsart}

\usepackage[a4paper, hmargin={2.8cm, 2.8cm}, vmargin={2.5cm, 2.5cm}]{geometry}
\usepackage[ngerman, english]{babel}
\usepackage[utf8]{inputenc}
\usepackage{pdfsync}
\usepackage{verbatim}
\usepackage[onehalfspacing]{setspace}
\usepackage{amsmath}
\usepackage{amsthm}
\usepackage{amssymb}
\usepackage{amsfonts}

 \usepackage{paralist}
\theoremstyle{theorem}
\newtheorem{satz}{Theorem}[section]

  \newtheorem{lemma}[satz]{Lemma}
  \newtheorem{kor}[satz]{Corollary}

  \newtheorem{prop}[satz]{Proposition}
  
  \theoremstyle{definition}
 \newtheorem{defi}[satz]{Definition}

    {\begin{proof}[Beweis]}
    {\end{proof}}
  \newtheorem{ex}[satz]{Example}

\newcommand{\norm}[1]{\lVert#1\rVert}   
\newcommand{\betrag}[1]{\lvert#1\rvert}
\newcommand{\KK}{\mathrm{KK}}

\newcommand{\Tor}{\mathrm{Tor}}
\newcommand{\K}{\mathrm K}

\newcommand{\RR}{\mathbb R}

\newcommand{\CC}{\mathbb C}
\newcommand{\NN}{\mathbb N}
\newcommand{\ZZ}{\mathbb Z}
\newcommand{\FF}{\mathbb F}
\newcommand{\TT}{\mathbb T}

\newcommand{\lk}{\langle}
\newcommand{\rk}{\rangle}
\newcommand{\id}{\text{id}}

\usepackage{tikz}
\usetikzlibrary{matrix,arrows}
\usepackage[pagebackref]{hyperref}
\allowdisplaybreaks

\usepackage{amsfonts}
\usepackage{amsmath}
\usepackage{amssymb}
\usepackage[utf8]{inputenc}
\usepackage{amsthm}
\usepackage{graphicx}
\usepackage{tikz}
\usepackage{tikz-cd}
\usepackage{hyperref}
\usepackage{amssymb}

\hypersetup{                    
    colorlinks=true,                
    breaklinks=true,                
    urlcolor= blue,                 
    linkcolor= blue,                
    citecolor= blue               
    }

\theoremstyle{definition}
\newtheorem{definition}{Definition}[section]
\newtheorem{thm}[definition]{Theorem}

\newcommand{\N}{\mathbb N}
\newcommand{\Z}{\mathbb Z}

\newcommand{\C}{\mathbb C}

\usepackage{geometry}
\usepackage{fancyhdr}

\title[Going-Down functors and the K\"unneth-formula]{\texorpdfstring{Going-Down functors and the Künneth formula for crossed products by étale groupoids}{Going-Down functors and the Künneth formula for crossed products by étale groupoids}}
\author{Christian B\"onicke$^{1}$}
\address{Mathematisches Institut der WWU M\"unster,
	\newline Einsteinstrasse 62, 48149 M\"unster, Germany}
\email{c.boenicke@wwu.de}

\author{Clément Dell'Aiera}
\address{Department of Mathematics, University of Hawaii
	\newline 2565 McCarthy Mall, Keller 401A Honolulu HI 96822.
}
\email{dellaiera@math.hawaii.edu}

\thanks{{$^{1}$} Supported by Deutsche Forschungsgemeinschaft (SFB 878).}

\begin{document}
\maketitle
\begin{abstract}
		We study the connection between the Baum-Connes conjecture for an ample groupoid $G$ with coefficient $A$ and the Künneth formula for the $\K$-theory of tensor products by the crossed product $A\rtimes_r G$. To do so we develop the machinery of Going-Down functors for ample groupoids.
		As an application we prove that both the uniform Roe algebra of a coarse space which uniformly embeds into a Hilbert space and the maximal Roe algebra of a space admitting a fibred coarse embedding into a Hilbert space satisfy the Künneth formula. We also provide a stability result for the K\"unneth formula using controlled $\K$-theory, and apply it to give an example of a space that does not admit a coarse embedding into a Hilbert space, but whose uniform Roe algebra satisfies the K\"unneth formula.
		As a by-product of our methods, we also prove a permanence property for the Baum-Connes conjecture with respect to equivariant inductive limits of the coefficient algebra.
\end{abstract}

\section{Introduction}
We say that a $\mathrm{C}^*$-algebra $A$ satisfies the \textit{Künneth formula} if for all $\mathrm{C}^*$-algebras $B$ there exists a canonical short exact sequence
	\begin{equation} \label{KunnethSequence}	
0\longrightarrow \K_*(A)\otimes \K_*(B)\stackrel{\alpha}{\longrightarrow}\K_*(A\otimes B)\stackrel{\beta}{\longrightarrow}\mathrm{Tor}(\K_*(A),\K_*(B))\longrightarrow 0,
	\end{equation}
	where $A\otimes B$ denotes the minimal tensor product of $A$ and $B$ and $\K_*$ denotes $\ZZ /2\ZZ$-graded $\K$-theory.

The Künneth formula is known to hold for every $\mathrm{C}^*$-algebra in the bootstrap class $\mathcal B$ by the results of Rosenberg and Schochet \cite{MR650021,RosenbergKunneth}. Recall, that $\mathcal B$ is the smallest class of separable nuclear $\mathrm{C}^*$-algebras such that:
\begin{itemize}
\item[$\bullet$] $\C \in \mathcal B$,
\item[$\bullet$] $\mathcal B$ is closed under countable inductive limits,  
\item[$\bullet$] $\mathcal B$ is closed under $\KK$-equivalence,
\item[$\bullet$] if $0 \rightarrow I \rightarrow A \rightarrow A/I \rightarrow 0$ is a short exact sequence of $C^*$-algebras and two of these are in $\mathcal B$, so is the third. 
\end{itemize}
In the groupoid setting, J-L. Tu proved in \cite[Lemma~10.6]{Tu98} that the reduced $\mathrm{C}^*$-algebra $C^*_r(G)$ satisfies the Künneth formula provided that $G$ is an a-T-menable groupoid. There are more $\mathrm{C}^*$-algebras which are known not to be in $\mathcal B$ but still satisfy the Künneth formula, such as reduced $\mathrm{C}^*$-algebras of lattices in $Sp(n,1)$. Indeed, if $A\in \mathcal B$, then $A$ is $\KK$-equivalent to a commutative $\mathrm{C}^*$-algebra (see \cite[Corollary~20.10.3]{MR1656031}). Moreover, a result of Skandalis \cite{SkandalisNotion} shows that if $\Gamma$ is an infinite hyperbolic property T group, then $C^*_r(\Gamma)$ is not $\K$-nuclear. In particular, it cannot be KK-equivalent to a commutative $C^*$-algebra (a more recent reference for this result is \cite[Theorem~6.2.1]{HigsonGuentnerNotes}) so that $C^*_r(\Gamma)$ is not in $\mathcal B$. But \cite[Corollary~0.2]{CEO} together with V. Lafforgue's result \cite{lafforgue2012conjecture} that hyperbolic groups satisfy the Baum-Connes conjecture with coefficients imply that $C^*_r(\Gamma)$ satisfies the Künneth formula.

A systematic study of the Künneth formula for crossed products by locally compact groups was undertaken in \cite{CEO}, where the following result was proved.
\begin{thm}\cite[Theorem~0.1~and~Corollary~0.2]{CEO} Let $G$ be a locally compact group and $A$ a $G$-algebra such that:
\begin{itemize}
\item[$\bullet$] $G$ satisfies the Baum-Connes conjecture with coefficients in all $C^*$-algebras $A\otimes B$ for all $C^*$-algebras $B$ with trivial $G$-action,
\item[$\bullet$] for every $C^*$-algebra $B$, considered as a $G$-algebra with trivial action, and every compact subgroup $K$ of $G$, $A\rtimes_r K$ satisfies the Künneth formula.
\end{itemize}
Then $A\rtimes_r G$ satisfies the Künneth formula.
\end{thm}

More recently, applying the newly developed methods of quantitative $\K$-theory, Oyono-Oyono and Yu were able to show that the uniform Roe algebra $C_u^*(X)$ satisfies the Künneth formula, provided that $X$ has finite asymptotic dimension \cite{OY4}.

In this paper we study the question of when $A$ satisfies the Künneth formula for the case that $A=C\rtimes_r G$ is a (reduced) crossed product, where $G$ is an ample groupoid and $C$ is a $G$-algebra.
	We follow the strategy of \cite{CEO} and compare existence of the sequence \ref{KunnethSequence} to the existence of a canonical exact sequence

		\begin{equation}\label{MixedSequence}
		0\rightarrow \K_*^{\mathrm{top}}(G;C)\otimes\K_*(B)\stackrel{\alpha_G}{\rightarrow}\K_*^{\mathrm{top}}(G;C\otimes B)\stackrel{\beta_G}{\rightarrow}\mathrm{Tor}(\K_*^{\mathrm{top}}(G;C),\K_*(B))\rightarrow 0\\
		\end{equation}

	Here $\K_*^{\mathrm{top}}(G;C)$ denotes the topological $\K$-theory of $G$ with coefficient $C$. The link between the sequences \ref{MixedSequence} and \ref{KunnethSequence} is given by the Baum-Connes assembly map $\mu_C:\K_*^{\mathrm{top}}(G;C)\rightarrow K_*(C\rtimes_r G)$.
	Let $\mathcal{N}_G$ denote the class of all separable exact $G$-algebras $C$ for which the canonical exact sequence \ref{MixedSequence} exists. We show that whenever $C$ is in $\mathcal{N}_G$ and $G$ satisfies the Baum-Connes conjecture with coefficients in $C\otimes B$ for all separable $\mathrm{C}^*$-algebras $B$ with respect to the trivial action on $B$, then $A=C\rtimes_r G$ satisfies the Künneth formula.
	We then use the machinery of Going-Down functors to show that the class $\mathcal{N}_G$ is non-empty, and in fact fairly large (see Theorem \ref{Theorem:Kunneth} and Corollary \ref{Cor:Kunneth}).
		As an immediate consequence of this and Proposition \ref{Prop:BCandKunneth} we can conclude that $A\rtimes_r G$ satisfies the Künneth formula for large classes of dynamical systems $(A,G,\alpha)$ (see Corollary \ref{Corollary:Kunneth}).
	We also show that the class $\mathcal{N}_G$ enjoys many stability properties. Among these we verify that $\mathcal{N}_G$ is stable under taking inductive limits. To prove this we show that the topological $\K$-theory is continuous with respect to the coefficient algebra (see Theorem \ref{Theorem:Continuity of top. K-theory}), which constitutes another application of the Going-Down principle and is inspired by \cite[§7]{MR1836047}:
	Another interesting consequence of this is a permanence property for the Baum-Connes conjecture, with respect to inductive limits of the coefficient algebra (see Corollary \ref{Cor:InductiveLimit}).
	
	We conclude section \ref{Section:MixedKunneth} by enlarging the class of groupoids our results can cover. While the Going-Down techniques require that we restrict ourselves to ample groupoids, we can extend the main results to cover many examples beyond that class. This is done by relating the classes $\mathcal{N}_G$ and $\mathcal{N}_H$, when $G$ and $H$ are equivalent groupoids on the one hand, and relating $\mathcal{N}_{G\ltimes X}$ and $\mathcal{N}_G$.
	In section \ref{Section:Applications} we take a look at several examples and applications, in particular we study the Künneth-formula for uniform (Theorem \ref{Theorem:Uniform Roe-algebra}) and maximal Roe algebras (Theorem \ref{Theorem: Maximal Roe-algebra}).
	
	Finally, in section \ref{Section:Stability} we use controlled $\K$-theory methods to inductively extend further the class of $\mathrm{C}^*$-algebras we can verify the Künneth formula for. In particular we give first examples of uniform Roe-algebras satisfying the Künneth formula for metric spaces not embedding coarsely into a Hilbert space.
	\color{black}

\section{Preliminaries on groupoids and $G$-algebras}
	Recall, that a \textit{groupoid} is a set $G$ together with a distinguished subset $G^{(2)}\subseteq G\times G$, called the set of \textit{composable pairs}, a product map $G^{(2)}\rightarrow G$ denoted by $(g,h)\mapsto gh$, and an inverse map $G\rightarrow G$, written $g\mapsto g^{-1}$, such that:
	\begin{enumerate}
		\item If $(g_1,g_2),(g_2,g_3)\in G^{(2)}$, then so are $(g_1g_2,g_3)$ and $(g_1,g_2g_3)$ and their products coincide, meaning $(g_1g_2)g_3=g_1(g_2g_3)$;
		\item for all $g\in G$ we have $(g,g^{-1})\in G^{(2)}$; and
		\item for any $(g,h)\in G^{(2)}$ we have $g^{-1}(gh)=h$ and $(gh)h^{-1}=g$.
	\end{enumerate}
	Every groupoid comes with a subset
	$$G^{(0)}=\lbrace gg^{-1}\mid g\in G\rbrace=\lbrace g^{-1}g\mid g\in G\rbrace$$
	called the set of \textit{units} of $G$, and two maps $r,d:G\rightarrow G^{(0)}$ given by
	$r(g)=gg^{-1}$ and $d(g)=g^{-1}g$ called \textit{range} and \textit{domain} maps respectively.
	A subgroupoid of $G$ is a subset $H\subseteq G$ which is closed under the product and inversion meaning that $gh\in H$ for all $(g,h)\in G^{(2)}\cap H\times H$ and $g^{-1}\in H$ for all $g\in H$.
	
	When $G$ is endowed with a locally compact Hausdorff topology under which the product and inversion maps are continuous, $G$ is called a locally compact groupoid. A \textit{bisection} is a subset $S\subseteq G$ such that the restrictions of the range and domain maps to $S$ are local homeomorphisms onto open subsets of $G$. We will denote the set of all open bisections by $G^{op}$. A locally compact, Hausdorff groupoid is called \textit{étale} if there is a basis for the topology of $G$ consisting of open bisections. It follows that $G^{(0)}$ is open in $G$. Recall that it is also closed, since $G$ is assumed to be Hausdorff. A topological groupoid is called \emph{ample} if it has a basis of compact open bisections. We will write $G^a$ for the subset of $G^{op}$ consisting of all compact open bisections. If $G$ is a locally compact, Hausdorff and étale groupoid, then $G$ is ample if and only if $G^{(0)}$ is totally disconnected (see \cite[Proposition 4.1]{Exel10}).
	
	For a subset $D\subseteq G^{(0)}$ write
	$$G_D:=\lbrace g\in G\mid d(g)\in D\rbrace,\ G^D:=\lbrace g\in G\mid r(g)\in D\rbrace,\ \text{and } G_D^D:=G_D\cap G^D.$$
	If $D=\lbrace u\rbrace$ consists of a single point $u\in G^{(0)}$ we will omit the braces in our notation and write $G_u:=G_D$, $G^u:=G^D$ and $G_u^u:=G_D^D$.
	
	Recall that if $X$ is a locally compact Hausdorff space and $A$ is a $\mathrm{C}^*$-algebra, then we call $A$ a $C_0(X)-algebra$ if there exists a non-degenerate $\ast$-homomorphism
	$$\Phi:C_0(X)\rightarrow Z(M(A)),$$ where $Z(M(A))$ denotes the center of the multiplier algebra of $A$. For every $x\in X$ there is a closed ideal $I_x$ in $A$ defined by $I_x=\overline{C_0(X\setminus\lbrace x\rbrace)A}$ and we call the quotient $A_x:=A/I_x$ the \textit{fibre} of $A$ over $x$. We write $a(x)$ for the image of $a\in A$ in $A_x$ under the quotient map. Put $\mathcal{A}=\coprod_{x\in X} A_x$. Then $\mathcal{A}$ can be equipped with a topology such that it becomes an upper-semicontinouos $\mathrm{C}^*$-bundle over $X$ and moreover $A\cong \Gamma_0(X,\mathcal{A})$, where $\Gamma_0(X,\mathcal{A})$ denotes the continuous sections of this bundle which vanish at infinity.
	Throughout this work we will freely alternate between the bundle picture and the picture as $C_0(X)$-algebras. For convenience bundles will always be denoted by calligraphic letters.
	The reader unfamiliar with the theory is referred to the expositions in \cite[Appendix C]{Williams} and \cite[Section~3.1]{Goehle}.
	
	Recall that a $\ast$-homomorphism $\Phi:A\rightarrow B$ between two $C_0(X)$-algebras $A$ and $B$ is called \textit{$C_0(X)$-linear} if $\Phi(f a)=f \Phi(a)$ for all $f\in C_0(X)$ and all $a\in A$.
	
	If $\Phi:A\rightarrow B$ is a $C_0(X)$-linear homomorphism, it induces $\ast$-homo\-morphisms $\Phi_x:A_x\rightarrow B_x$ on the level of the fibres given by $\Phi_x(a(x))=\Phi(a)(x)$.
	Conveniently, one can check several properties of $\Phi$ on the level of the fibres and vice versa:
	\begin{lemma}\cite[Lemma~2.1]{MR2820377}\label{Lem:IsomorphismCriteriumForC(X)-linearHomomorphisms}
		Let $\Phi:A\rightarrow B$ be a $C_0(X)$-linear homomorphism. Then $\Phi$ is injective (resp. surjective, resp. bijective) if and only if $\Phi_x$ is injective (resp. surjective, resp. bijective) for all $x\in X$.
	\end{lemma}
	
	We will also need the notion of a pullback: If $A$ is a $C_0(X)$-algebra and $f:Y\rightarrow X$ a continuous map, we can define the \textit{pullback} of $A$ along $f$ as follows:
	Let $q:\mathcal{A}\rightarrow X$ denote the upper-semicontinouos $\mathrm{C}^*$-bundle over $X$ associated to $A$. Then we can form the pullback bundle $f^*\mathcal{A}=\lbrace ((y,a)\in Y\times\mathcal{A}\mid f(y)=q(a)\rbrace$. The bundle $f^*\mathcal{A}$ is an upper-semicontinouos $\mathrm{C}^*$-bundle over $Y$ whose fibres $(f^*\mathcal{A})_y$ are canonically isomorphic to $A_{f(y)}$. We let $f^*A:=\Gamma_0(Y,f^*\mathcal{A})$ denote the corresponding $C_0(Y)$-algebra. Note, that we can canonically identify $(f^*A)_y=A_{f(y)}$.
	It is an easy exercise to show that if $A$ is a $C_0(X)$-algebra and $f:Y\rightarrow X$ and $g:Z\rightarrow Y$ are two continuous maps, then the algebras $(f\circ g)^*A$ and $g^*(f^*A)$ are canonically isomorphic as $C_0(Z)$-algebras.
	
	Pullbacks also behave nicely with respect to $C_0(X)$-linear $\ast$-homomorphisms:	
	\begin{lemma}\label{Lem:PullbackOfHomomorphisms}
		Let $A$ and $B$ be two $C_0(X)$-algebras and $f:Y\rightarrow X$ a continuous map. If $\Phi:A\rightarrow B$ is a $C_0(X)$-linear homomorphism, then the map
		$$f^*\Phi:f^*A\rightarrow f^*B$$
		given by $(f^*\Phi)(\psi)(y)=\Phi_{f(y)}(\psi(y))$ is a $C_0(Y)$-linear homomorphism.
		Moreover, the pullback construction is functorial meaning if $\Psi:B\rightarrow C$ is another $C_0(X)$-linear $*$-homo\-morphism into a $C_0(X)$-algebra $C$ then $f^*\Psi\circ f^*\Phi=f^*(\Psi\circ \Phi)$.
	\end{lemma}
	
	Recall that a \textit{groupoid dynamical system} $(A,G,\alpha)$ consists of a locally compact Hausdorff groupoid $G$, a $C_0(G^{(0)})$-algebra $A$ and a family $(\alpha_g)_{g\in G}$ of $*$-isomorphisms $\alpha_g:A_{d(g)}\rightarrow A_{r(g)}$ such that $\alpha_{gh}=\alpha_g\circ \alpha_h$ for all $(g,h)\in G^{(2)}$ and such that $g\cdot a:=\alpha_g(a)$ defines a continuous action of $G$ on the upper-semicontinuous bundle $\mathcal{A}$ associated to $A$.	
	We will often omit the action $\alpha$ in our notation and just say that $A$ is a $G$-algebra.
	Since the topology on an upper-semicontinuous $\mathrm{C}^*$-bundle is notoriously difficult to handle we will rely on the following alternate characterization in this paper:
	\begin{lemma}\cite[Lemma~4.3]{MR2547343}
		Let $(A,G,\alpha)$ be a groupoid dynamical system. Then the mapping $$f\mapsto [g\mapsto \alpha_g(f(g))]$$ defines a $C_0(G)$-linear $\ast$-isomorphism $d^*A\rightarrow r^*A$, also denoted by $\alpha$.
		
		Conversely, if $G$ is a groupoid, $A$ a $C_0(G^{(0)})$-algebra, and $\alpha:d^*A\rightarrow r^*A$ is a $C_0(G)$-linear isomorphism then $\alpha$ induces an isomorphism $\alpha_g:A_{d(g)}\rightarrow A_{r(g)}$ for each $g\in G$. If the equation $\alpha_{gh}=\alpha_g\alpha_h$ holds for all $(g,h)\in G^{(2)}$, then $(A,G,\alpha)$ is a groupoid dynamical system.
	\end{lemma}
	
	Finally, let us briefly recall the definition of a groupoid crossed product	following \cite{MR1900993}.
	Let $G$ be an étale groupoid and $(A,G,\alpha)$ a groupoid dynamical system. Consider the complex vector space $\Gamma_c(G,r^*\mathcal{A})$. It carries a canonical $*$-algebra structure with respect to the following operations:
	$$(f_1\ast f_2)(g)=\sum\limits_{h\in G^{r(g)}} f_1(h)\alpha_h(f_2(h^{-1}g))$$
	and
	$$f^*(g)=\alpha_g(f(g^{-1})^*).$$
	See for example \cite[Proposition~4.4]{MR2547343} for a proof of this fact.
	For $u\in G^{(0)}$ consider the Hilbert $A_u$-module $\ell^2(G^u,A_u)$. It is the completion of the space of finitely supported $A_u$-valued functions on $G^u$, with respect to the inner product 
	$$\lk \xi,\eta\rk=\sum\limits_{h\in G^u}\xi(h)^*\eta(h).$$
	We can then define a $*$-representation $\pi_u:\Gamma_c(G,r^*\mathcal{A})\rightarrow \mathcal{L}(\ell^2(G^u,A_u))$ by
	$$\pi_u(f)\xi(g)=\sum\limits_{h\in G^u}\alpha_g(f(g^{-1}h))\xi(h).$$
	Using this family of representations, we can define a $\mathrm{C}^*$-norm on the convolution algebra $\Gamma_c(G,r^*\mathcal{A})$ by
	$$\norm{f}_r:=\sup\limits_{u\in G^{(0)}}\norm{\pi_u(f)}.$$
	The reduced crossed product $A\rtimes_r G$ is defined to be the completion of $\Gamma_c(G,r^*\mathcal{A})$ with respect to $\norm{\cdot}_r$.

\section{Inductive limits of $G$-algebras}
In this section we will show that an inductive limit of $G$-algebras with $G$-equivariant connecting morphisms is again a $G$-algebra in a canonical fashion. These results should be known to the experts but since we could not find a suitable reference and in order to keep the exposition self-contained we elaborate on the details. We start of by considering $C_0(X)$-algebras:
	Let $(A_i,\varphi_{i,j})$ be an inductive system of $\mathrm{C}^*$-al\-gebras, where each $A_i$ is a $C_0(X)$-algebra, such that the connecting homomorphisms $\varphi_i$ are $C_0(X)$-linear.
	If $A=\lim\limits_{\rightarrow}A_i$, then $A$ is a $C_0(X)$-algebra in a canonical way:
	
	Let us start by recalling the construction of the limit algebra $A$:
	Consider the algebra $$\widetilde{A}=\lbrace (a_i)_i\in\prod\limits_{i\in I}A_i\mid \exists i_0: a_{i}=\varphi_{i,i_0}(a_{i_0})\forall i\geq i_0\rbrace.$$
	Then $A$ is the closure of the image of $\widetilde{A}$ under the quotient map $q:\prod A_i\rightarrow \prod A_i/\bigoplus A_i$.
	Now if $f\in C_0(X)$, then $C_0(X)$-linearity of the $\varphi_{i,j}$ implies, that $\widetilde{A}$ is invariant under component-wise multiplication with $f$. It also leaves the ideal $\bigoplus A_i$ invariant. Hence we get a well-defined linear map
	$q(\widetilde{A})\rightarrow q(\widetilde{A})$ by $f\cdot q((a_i)_i):=q((f\cdot a_i)_i)$. Using the equality $\norm{q((a_i)_i)}=\lim \norm{a_i}$ we get
	$\norm{q((f\cdot a_i)_i)}=\lim \norm{f\cdot a_i}\leq \norm{f}\lim \norm{a_i}=\norm{f}\norm{q((a_i)_i)}$. Consequently, $f\cdot$ extends to a bounded linear map $A\rightarrow A$, actually to an element in $Z(M(A))$, where the adjoint is given by $\overline{f}\cdot$. Thus, we have constructed a $*$-homomorphism $\Phi:C_0(X)\rightarrow Z(M(A))$. 
	\begin{lemma}\label{Lem:InductiveLimitsOfC_0(X)-algebras}
		The $*$-homomorphism $\Phi$ from above is non-degener\-ate. Consequently, $A$ is a $C_0(X)$-algebra such that the canonical maps $\psi_{i}:A_i\rightarrow A$ are $C_0(X)$-linear.
	\end{lemma}
	\begin{proof}
		Let $a\in A$ and $\varepsilon>0$ be given. By construction of the inductive limit $\bigcup_{i\in I} \psi_i(A_i)$ is dense in $A$, so there exists an $i\in I$ and $b\in A_i$ such that $\norm{\psi_i(b)-a}<\frac{\varepsilon}{2}$. Since the structure homomorphism for $A_i$ is non-degenerate we can also find $f\in C_0(X)$ and $c\in A_i$ such that $\norm{b-fc}< \frac{\varepsilon}{2\norm{\psi_i}}$, and hence $\norm{\psi_i(b)-f\psi_i(c)}<\frac{\varepsilon}{2}$.
		Combining the above inequalites we obtain
		$\norm{f\psi_i(c)-a}<\norm{f\psi_i(c)-\psi_i(b)}+\norm{\psi_i(b)-a}<\varepsilon$.
	\end{proof}
	
	We will now identify the fibres of the limit algebra:
	\begin{lemma}
		Let $(A_i,\varphi_{i,j})$ be an inductive system of $C_0(X)$-al\-gebras and $A=\lim_{i} A_i$. Then, for every $x\in X$, $((A_i)_x,(\varphi_{i,j})_x)$ is an inductive system of $\mathrm{C}^*$-algebras and
		$$ \lim\limits_{i} (A_i)_x\cong A_x.$$
	\end{lemma}
	\begin{proof}
		It is immediate, that $((A_i)_x,(\varphi_{i,j})_x)$ is indeed an inductive sequence of $\mathrm{C}^*$-algebras. Hence we only need to identify the limit.
		Let $\pi_{i,x}:A_i\rightarrow (A_i)_x$ denote the quotient maps onto the fibres and $\psi_{i,x}:(A_i)_x\rightarrow \lim\limits_{i}(A_i)_x$ the canonical maps. By the universal property of the limit we obtain a surjective $*$-homo\-morphism $$\pi:A\rightarrow \lim\limits_{i}(A_i)_x.$$
		It remains to show that the kernel of $\pi$ coincides with the ideal $I_x=\overline{C_0(X\setminus\lbrace x\rbrace)A}$ of $A$.
		If $a=\psi_i(b)$ for some $b\in A_i$ and $f\in C_0(X\setminus\lbrace x\rbrace)$, then $\pi(fa)=\pi(f\psi_i(b))=\pi(\psi_i(fb))=\psi_{i,x}(\pi_{i,x}(fb))=0$. By continuity we get $I_x\subseteq ker(\pi)$.
		
		Suppose conversely that $a\in ker(\pi)$ and $\varepsilon>0$ is given. First we can find $i\in I$ and $b\in A_i$ such that $\norm{a-\psi_i(b)}<\frac{\varepsilon}{3}$. Thus,
		$\norm{\psi_{i,x}(\pi_{i,x}(b))}=\norm{\pi(\psi_i(b))}=\norm{\pi(\psi_i(b)-a)}\leq \norm{a-\psi_i(b)}<\frac{\varepsilon}{3}$.
		Upon replacing $b$ and $i$ by $\varphi_{j,i}(b)$ for $j\geq i$ big enough we can actually assume that $\norm{\pi_{x,i}(b)}<\frac{\varepsilon}{3}$. Then there exists some $b'\in A_i$ such that $\norm{b-b'}<\frac{\varepsilon}{3}$ and $\pi_{i,x}(b')=0$. Hence there must be $b''\in A_i$ and $\varphi\in C_0(X\setminus\lbrace x\rbrace)$ such that $\norm{b'-\varphi b''}<\frac{\varepsilon}{3}$.
		Putting things together we obtain
		$$\norm{a-\varphi\psi_i(b'')}\leq \norm{ a- \psi_i(b)}+\norm{\psi_i(b)-\psi_i(b')}+\norm{\psi_i(b')-\psi_i(\varphi b'')}<\varepsilon$$
		and hence $ker(\pi)\subseteq I_x$, which completes the proof.
	\end{proof}
	Next, we want to show that taking the limit of an inductive system commutes with pullbacks: Let $(A_i,\varphi_{i,j})$ be an inductive system of $C_0(X)$-algebras and $f:Y\rightarrow X$ a continuous map. Then we get $C_0(Y)$-linear $*$-homo\-morphisms $f^*\varphi_{i,j}:f^*A_j\rightarrow f^* A_{i}$ by the formula $$(f^*\varphi_{i,j})(\xi)(y)=(\varphi_{i,j})_{f(y)}(\xi(y)).$$
	as in Lemma \ref{Lem:PullbackOfHomomorphisms}. 
	\begin{prop}\label{Prop:LimitsAndPullbacks}
		Let $(A_i,\varphi_{i,j})$ be an inductive system of $C_0(X)$-al\-gebras and $f:Y\rightarrow X$ a continuous map.
		Then $(f^*A_i,f^*\varphi_{i,j})$ is an inductive system of $C_0(Y)$-algebras and $f^*(\lim_i A_i)$ is $C_0(Y)$-linearly isomorphic to $\lim_i f^*(A_i)$.
	\end{prop}
	\begin{proof}
		Let $A=\lim_i A_i$ and $\psi_i:A_i\rightarrow A$ be the canonical $\ast$-homo\-morphisms. Then by Lemma \ref{Lem:PullbackOfHomomorphisms} we obtain $C_0(Y)$-linear $*$-homo\-morphisms $f^*\psi_i:f^*A_i\rightarrow f^*A$ such that $f^*\psi_{i}\circ f^*\varphi_{i,j}=f^*(\psi_{i}\circ \varphi_{i,j})=f^*\psi_j$. Using the universal property of the limit, we obtain a $C_0(Y)$-linear $*$-homo\-morphism $$\Psi:\lim\limits_{i}f^*A_i\rightarrow f^*A.$$
		To show that it is an isomorphism, it is enough to check that $\Psi_y$ is an isomorphism for all $y\in Y$. But under the identifications
		$$(\lim_{i}f^*A_i)_y\cong \lim_{i}(A_i)_{f(y)}\text{ and }(f^*A)_y\cong A_{f(y)}$$ the map $\Psi_y$ coincides with the isomorphism 
		$$\lim\limits_{i}(A_i)_{f(y)}\rightarrow A_{f(y)}$$
		from the previous Lemma.
	\end{proof}
	
	Suppose now that $(A_i,\varphi_{i,j})$ is an inductive system of $G$-alge\-bras, such that all the connecting homomorphisms are $G$-equivariant. We have already seen in Lemma \ref{Lem:InductiveLimitsOfC_0(X)-algebras}, that $A=\lim_i A_i$ is a $C_0(G^{(0)})$-algebra in a canonical way, such that all the homo\-morphisms $\psi_i:A_i\rightarrow A$ are $C_0(G^{(0)})$-linear. The following Proposition shows how we can use the $G$-actions at each stage of the sequence to obtain a $G$-action on the limit.
	\begin{prop}
		Let $(A_i,\varphi_{i,j})$ be an inductive system of $G$-algebras, such that $\varphi_{i,j}$ is $G$-equivariant for all $i,j\in I$ with $i\geq j$. Let $A:=\lim_i A_i$ and $\psi_i:A_i\rightarrow A$ be the canonical maps. Then there exists a canonical $G$-action on $A$, such that $\psi_i$ is $G$-equivariant for all $i\in I$.
	\end{prop}
	\begin{proof}
		For each $i\in I$ let $\alpha_i:d^*A_i\rightarrow r^*A_i$ denote the $C_0(G)$-linear isomorphism implementing the action of $G$ on $A_i$. Since $\varphi_{i,j}$ is $G$-equivariant for all $i,j\in I$ with $i\geq j$ we have commutative diagrams
		\begin{center}
			\begin{tikzpicture}[description/.style={fill=white,inner sep=2pt}]
			\matrix (m) [matrix of math nodes, row sep=3em,
			column sep=2.5em, text height=1.5ex, text depth=0.25ex]
			{ d^*A_j &  r^*A_j\\
				d^*A_{i} &  r^*A_{i}\\ };
			\path[->,font=\scriptsize]
			(m-1-1) edge node[auto] {$ \alpha_j $} (m-1-2)
			(m-2-1) edge node[auto] {$ \alpha_{i} $} (m-2-2)
			(m-1-1) edge node[auto] { $ d^*\varphi_{i,j} $ } (m-2-1)
			(m-1-2) edge node[auto] {$ r^*\varphi_{i,j} $} (m-2-2)
			;
			\end{tikzpicture}
		\end{center}
		By the universal property, we obtain a $C_0(G)$-linear $*$-isomorphism between the respective limits. Combining this with Proposition \ref{Prop:LimitsAndPullbacks} we obtain a $C_0(G)$-linear $*$-isomorphism
		$$\alpha:d^*A\rightarrow r^*A.$$
		As each $\alpha_i$ is compatible with the multiplication in $G$, so is the limit homomorphism $\alpha$.
	\end{proof}
	
\section{Going-Down Functors}
	We would like to use the Going-Down principle as developed in \cite{1806.00391}. Although \cite[Theorem~7.10]{1806.00391} can be applied directly in many situations, oftentimes it is not directly a map on $\K_*^{\mathrm{top}}(G;A)$ one is interested in, but a map on a construction involving this group, which still shares the same basic functorial properties.
	Moreover, the map in question must not necessarily be given by taking the Kasparov product. A closer inspection of the proof of \cite[Theorem~7.10]{1806.00391} reveals, that we only used the naturality of the Kasparov product. Hence, following \cite{CEO} we can use the language of category theory to obtain a more general result.
	To begin with, given a second countable ample groupoid $G$, we denote by $\mathcal{C}(G)$ the category of separable commutative proper $G$-algebras, i.e. algebras of the form $C_0(X)$, where $X$ is a second countable proper $G$-space. Also let $\mathcal{S}(G)$ be the set containing $G$ and all of its compact open subgroupoids.
	\begin{defi}\label{Def:GDfunctor} Let $G$ be an ample groupoid.
		A \textit{Going-Down functor} for $G$ is a collection of $\ZZ$-graded functors $\mathcal{F}=(\mathcal{F}^n_H)_{H\in \mathcal{S}(G)}$, where $\mathcal{F}^n_{H}$ is a covariant additive functor from the category of second countable, proper, locally compact $G$-spaces (with morphisms being the proper, continuous $G$-maps) to the category of abelian groups, such that the following axioms are satisfied:
		\begin{enumerate}
			\item Cohomology axioms: For every $H\in \mathcal{S}(G)$
			\begin{enumerate}
				\item the functor $\mathcal{F}_H^n$ is homotopy invariant;
				\item the functor $\mathcal{F}_H^n$ is half-exact, i.e. for every short exact sequence $$0\longrightarrow I\longrightarrow A\longrightarrow A/I\longrightarrow 0$$
				in $\mathcal{C}(H)$, the sequence
				$$\mathcal{F}_H^n (A/I)\longrightarrow \mathcal{F}_H^n(A)\longrightarrow\mathcal{F}_H^n$$
				is exact in the middle; and 
				\item for each $n\in\ZZ$ there is a natural equivalence between $\mathcal{F}_H^{n+1}$ and the functor $A\mapsto \mathcal{F}_H^n(A\otimes C_0(\RR))$, where $H$ acts trivially on the second tensor factor.
			\end{enumerate}
			\item Induction axiom: For every compact open subgroupoid $H$ of $G$, there are natural equivalences $I_H^G(n)$ between the functors $\mathcal{F}_H^n$ and $\mathcal{F}_G^n\circ Ind_H^G$, compatible with suspension,
			where $Ind_H^G:\mathcal{C}(H)\rightarrow \mathcal{C}(G)$, $A\mapsto Ind_H^{G_{\mid H^{(0)}}} A$ denotes induction from $H$-algebras to $G$-algebras (see \cite[Section~3]{1806.00391} for a detailed discussion of induction for groupoids).\color{black}
		\end{enumerate}
		If $\mathcal{F}$ is a Going-Down functor for $G$, we define $$\mathcal{F}^n(G):=\lim\limits_{X\subseteq \mathcal{E}(G)}\mathcal{F}^n_G(C_0(X)),$$
		where $X$ runs through the $G$-compact subsets of $\mathcal{E}(G)$.
	\end{defi}
	Our main examples of Going-Down functors arise from the topological $\K$-theory of ample groupoids:
	\begin{ex}\label{Example:Going-Down functor}
		Let $G$ be a second countable ample groupoid and $A$ be a fixed $G$-algebra. Define $\mathcal{F}_H^*(C_0(X)):=\mathrm{KK}^H_*(C_0(X),A_{\mid H})$ for $H\in\mathcal{S}(G)$ and $C_0(X)\in\mathcal{C}(H)$, where $\mathrm{KK}^H$ denotes Le Gall's groupoid equivariant $\mathrm{KK}$-theory (see \cite{LeGall}). Then $\mathcal{F}$ is a $\ZZ/2\ZZ$-graded Going-Down functor:
		\begin{enumerate}
			\item Cohomology axioms:
			\begin{enumerate}
				\item Homotopy invariance is clear, since groupoid equivariant $\mathrm{KK}$-theory is invariant with respect to equivariant homotopies in the first variable.
				\item Half-exactness follows from \cite[Proposition~7.2 and Lemma~7.7]{Tu99}.
				\item The suspension axiom is clear from the definition of the higher equivariant $\mathrm{KK}$-groups.
			\end{enumerate}
			\item The natural equivalence required in the induction axiom is provided by the compression homomorphism defined prior to \cite[Theorem~6.2]{1806.00391} or rather its inverse, the inflation map (see also \cite[Lemma~5.2.6]{DellAieraThesis}). From the definition of the compression homomorphism it is easy to see, that it indeed provides a natural transformation with respect to equivariant $\ast$-homomorphisms.
		\end{enumerate}
	\end{ex}
	
	The following lemma can be proved using standard homotopy techniques (see for example \cite[§21.4]{MR1656031})
	\begin{lemma}\label{Lem:LES} Let $\mathcal{F}$ be a Going-Down functor.
		For every short exact sequence $$0\longrightarrow I\longrightarrow A\longrightarrow A/I\longrightarrow 0$$ in $\mathcal{C}(H)$ there are natural maps $\partial_n:\mathcal{F}_H^n(I)\rightarrow\mathcal{F}_H^{n+1}(A/I)$ providing a long exact sequence
		$$\cdots\longrightarrow \mathcal{F}_H^n(A/I)\longrightarrow \mathcal{F}_H^{n}(A)\longrightarrow\mathcal{F}_H^n(I)\stackrel{\partial_n}{\longrightarrow}\mathcal{F}_H^{n+1}(A/I)\longrightarrow\cdots$$
	\end{lemma}
	\begin{defi}
		Let $\mathcal{F}$ and $\mathcal{G}$ be Going-Down functors for the ample groupoid $G$. A \textit{Going-Down transformation} is a collection $\Lambda=(\Lambda_H^n)_{H\in \mathcal{S}(G)}$ of natural transformations between $\mathcal{F}_H^n$ and $\mathcal{G}_H^n$ compatible with suspension, such that
		$I_H^G(n)\circ \Lambda_H^n=\Lambda_G^n\circ I_H^G(n)$.
	\end{defi}
	\begin{ex}
		Let $G$ be a second countable ample groupoid and $A$ and $B$ be separable $G$-algebras. Let $\mathcal{F}$ be the Going-Down functor defined by $\mathcal{F}_H^*(C_0(X))=\mathrm{KK}^H_*(C_0(X),A_{\mid H})$ and let $\mathcal{G}$ be the Going-Down functor defined by $\mathcal{G}_H^*(C_0(X))=\mathrm{KK}^H_*(C_0(X),B_{\mid H})$ as in Example \ref{Example:Going-Down functor}. Suppose that $x\in \mathrm{KK}^G(A,B)$. Then we can define a Going-Down transformation
		$\Lambda$ from $\mathcal{F}$ to $\mathcal{G}$ by letting $\Lambda_H^*(C_0(X))$ be the map $$\mathcal{F}_H^*(C_0(X))=\mathrm{KK}^H_*(C_0(X),A_{\mid H})\stackrel{\cdot \otimes x}{\rightarrow} \mathrm{KK}^H_*(C_0(X),B_{\mid H})=\mathcal{G}_H^*(C_0(X)).$$
		By associativity of the Kasparov product, $\Lambda_H^*$ is a natural transformation, which is clearly compatible with suspension. Compatibility with $I_H^G$ follows from \cite[Lemma~6.7]{1806.00391}.
	\end{ex}
	Using the naturality, a Going-Down transformation $\Lambda$ between two Going-Down functors $\mathcal{F}$ and $\mathcal{G}$ induces morphisms $\Lambda^n(G):\mathcal{F}^n(G)\rightarrow\mathcal{G}^n(G)$ in the limit.
	
	\begin{satz}\label{Theorem:Going-Down Theorem}
		Let $\mathcal{F}$ and $\mathcal{G}$ be two Going-Down functors for an ample groupoid $G$ and let $\Lambda$ be a Going-Down transformation between $\mathcal{F}$ and $\mathcal{G}$. Suppose that $\Lambda_H^n(C(H^{(0)})):\mathcal{F}_H^n(C(H^{(0)}))\rightarrow\mathcal{G}_H^n(C(H^{(0)}))$ is an isomorphism for all compact open subgroupoids $H$ of $G$. Then $\Lambda^n(G):\mathcal{F}^n(G)\rightarrow\mathcal{G}^n(G)$ is an isomorphism.
	\end{satz}
	\begin{proof} The proof is essentially the same as that of \cite[Theorem~7.10]{1806.00391}, replacing the group $\mathrm{KK}^H_*(C_0(X),A_{\mid H})$ by $\mathcal{F}^*_H(C_0(X))$ and $\mathrm{KK}^H_*(C_0(X),B_{\mid H})$ by $\mathcal{G}^*_H(C_0(X))$, and the map $\cdot \otimes res_H^G(x)$ by $\Lambda_H^*$. For convenience of the reader we recall the main steps in the proof:
	By definition of $\Lambda^n(G)$ we have to show that $\Lambda^n_G(C_0(X))$ is an isomorphism for all $G$-compact proper $G$-spaces $X\subseteq \mathcal{E}(G)$. This is done by reducing to more and more special spaces $X$. For our exposition we reverse the order and start with the most special situation:
	
	Suppose that $X=GU$, where $U\subseteq X$ is a compact open subset of $X$ such that the anchor map $p:X\rightarrow G^{(0)}$ restricts to a homeomorphism from $U$ onto the (compact) open set $p(U)\subseteq G^{(0)}$.
	Then it is easy to see that $H=\lbrace g\in G\mid gU\cap U\neq \emptyset\rbrace$ is a compact open subgroupoid of $G$ and there is a canonical $G$-equivariant homeomorphism $G\times_H U\cong X$. Using the fact that $Ind_H^G(C_0(U))=C_0(G\times_H U)$ and that the natural equivalence $I_H^G(n)$ is compatible with $\Lambda$ in the sense that $I_H^G(n)\circ \Lambda_H^n=\Lambda_G^n\circ I_H^G(n)$, we conclude that $\Lambda^n_G(C_0(X))$ is an isomorphism in this case.
	
	Next, consider the case, where every point $x\in X$ admits a compact open neighbourhood $U$ such that the anchor map $p:X\rightarrow G^{(0)}$ restricts to a homeomorphism from $U$ onto the (compact) open set $p(U)\subseteq G^{(0)}$. Using $G$-compactness of $X$ we can write $X=\bigcup_{i=1}^n GU_i$, where each $U_i$ is a compact open set with the property above. Then we can inductively apply the Mayer-Vietoris sequences for $\mathcal{F}$ and $\mathcal{G}$ respectively and compare them via $\Lambda$. Since two out of three arrows in the resulting diagram are isomorphisms by the first step above, an application of the Five-Lemma completes the proof in this case.
	
	Next, one realizes that $X$ as in the second step above is an instance of a zero-dimensional $G$-simplicial complex in the sense of \cite[Definition~7.5]{1806.00391}. Hence we can proceed by induction on the dimension $n$ of a (typed) $G$-simplicial complex, to show the claim for all finite dimensional $G$-simplicial complexes. Let $X$ be a (typed) $G$-simplicial complex of dimension $n>0$,  $Y$ be its $(n-1)$-skeleton, and $U=X\setminus Y$ the union of all open $n$-simplices.
	Then we get an exact sequence of $G$-algebras
	$$0\longrightarrow C_0(U)\longrightarrow C_0(X)\longrightarrow C_0(Y)\longrightarrow 0.$$
	Applying the corresponding long exact sequences for $\mathcal{F}$ and $\mathcal{G}$ from Lemma \ref{Lem:LES} and comparing them with $\Lambda$, our claim will follow from the Five-Lemma, once we show that $\Lambda^n_G(C_0(U))$ is an isomorphism. 
	But $U$ is equivariantly homeomorphic to $X'\times \RR^n$, where $X'$ denotes the barycenters of $n$-dimensional simplices. Thus, we have $\mathcal{F}^m_G(C_0(U))\cong \mathcal{F}^{m+n}_G(C_0(X'))$. Since $\Lambda$ is compatible with suspensions, it is enough to show that $\Lambda_G^m(C_0(X'))$ is an isomorphism for all $m\in \ZZ$. But $X'$ is a $G$-compact, proper $G$-space such that every point $x\in X'$ admits a compact open neighbourhood $U$ such that the anchor map $p:X'\rightarrow G^{(0)}$ restricts to a homeomorphism from $U$ onto the (compact) open set $p(U)\subseteq G^{(0)}$. Hence we reduced everything to the second step above.
	
	Finally, it follows from \cite[Lemma~7.7]{1806.00391} and \cite[Proposition~3.2]{Tu12} that every $G$-compact proper $G$-space admits a $G$-equivariant continuous map into a finite dimensional $G$-simplicial complex. Together with the universal property of $\mathcal{E}(G)$ and a zig-zag-type argument it follows that $\Lambda^m_G(C_0(X))$ is an isomorphism for arbitrary $G$-compact proper subspaces $X\subseteq \mathcal{E}(G)$.
	\end{proof}

\section{Continuity of Topological K-theory}
In this section we will show, that the topological $\K$-theory of an ample groupoid is continuous with respect to the coefficient algebra.
	Recall, that an étale groupoid is called \textit{exact}, if for every $G$-equivariant exact sequence
	$$0\rightarrow I\rightarrow A\rightarrow B\rightarrow 0$$
	of $G$-algebras, the corresponding sequence
		$$0\rightarrow I\rtimes_r G\rightarrow A\rtimes_r G\rightarrow B\rtimes_r G \rightarrow 0$$
		of reduced crossed products is exact.
	The following is an analogue of \cite[Lemma~2.5]{MR2010742} for étale groupoids:
	\begin{lemma}\label{Lemma:Proper Groupoids and inductive limits}
		Let $G$ be an étale groupoid and $(A_n,\varphi_n)_n$ an inductive sequence of $G$-algebras with limit $A=\lim_n A_n$. Then $(A_n\rtimes_r G,\varphi_n\rtimes G)_n$ is an inductive sequence of $\mathrm{C}^*$-algebras. Suppose additionally, that either one of the following conditions hold:
		\begin{enumerate}
			\item All the connecting maps $\varphi_n$ are injective.
			\item The groupoid $G$ is exact.
		\end{enumerate}
		Then $A\rtimes_r G=\lim_n A_n\rtimes_r G$ with respect to the connecting homomorphisms $\varphi_n\rtimes G$.
	\end{lemma}
	\begin{proof}
		It is clear that $(A_n\rtimes_r G,\varphi_n\rtimes G)$ is an inductive sequence of $\mathrm{C}^*$-algebras. For the second statement we follow the argument in \cite[Lemma~2.5]{MR2010742}:
		In the case of $(1)$ we may regard each $A_n\rtimes_r G$ as a subalgebra of $A\rtimes_r G$ and hence also the inductive limit $\overline{\bigcup_{n\in\NN}A_n\rtimes_r G}$ is contained in $A\rtimes_r G$. Let us check that $\bigcup_{n\in\NN} \Gamma_c(G,r^*\mathcal{A}_n)\subseteq \overline{\bigcup_{n\in\NN}A_n\rtimes_r G}$ is dense in $A\rtimes_r G$. First, consider elements of the form $f\otimes a\in \Gamma_c(G,r^*\mathcal{A})$ for $f\in C_c(G)$ and $a\in A$. Let $\varepsilon>0$ be given. Then, by $(1)$ we can find $n\in\NN$ and $b\in A_n$ such that $\norm{a-b}<\frac{\varepsilon}{\norm{f}}$. It follows, that $f\otimes b\in \Gamma_c(G,r^*\mathcal{A}_n)$ with $\norm{f\otimes a-f\otimes b}\leq \norm{f}\norm{a-b}<\varepsilon$. Since finite sums of elements of the form $f\otimes a$ are dense in $\Gamma_c(G,r^*\mathcal{A})$ in the inductive limit topology, it follows that $\bigcup_{n\in\NN}\Gamma_c(G,r^*\mathcal{A}_n)$ is dense in $\Gamma_c(G,r^*\mathcal{A})$ with respect to the inductive limit topology and hence also with respect to the reduced norm topology.
		For the proof of $(2)$ we make use of the following general fact:
		If $(B_n,\psi_n)$ is an inductive sequence of $\mathrm{C}^*$-algebras, then so is $(B_n/ker(\psi_n),\widetilde{\psi_n})$, where $\widetilde{\psi}_n$ are the maps induced by $\psi_n$ on the quotients. Then it is easy to check, that all the maps $\widetilde{\psi}_n$ are injective and $B=\lim_n B_n/ker(\psi_n)$.
		Returning to the proof of $(2)$,
		let $I_n=ker(\varphi_n)$. Using the exactness of $G$ now, we see that $I_n\rtimes G$ is precisely the kernel of the map $\varphi_n \rtimes G:A_n\rtimes_r G\rightarrow A\rtimes_r G$.
		By the above remark we have $A=\lim_n A_n/I_n$ and since the connecting maps are all injective we get $A\rtimes_r G=\lim_n (A_n/I_n)\rtimes_r G$ by $(1)$. Using the exactness of $G$ now, we see that $I_n\rtimes G$ is precisely the kernel of the map $\varphi_n \rtimes G:A_n\rtimes_r G\rightarrow A\rtimes_r G$, hence $(A_n/I_n)\rtimes_r G=A_n\rtimes_r G/I_n\rtimes_r G$ and another application of the above mentioned fact together with the identity $I_n\rtimes_r G=ker(\varphi_n\rtimes_r G)$ yields
		$$A\rtimes_r G	=\lim_n A_n/I_n\rtimes_r G=\lim_n A_n\rtimes_r G/I_n\rtimes_r G=\lim_n A_n\rtimes_r G.$$
	\end{proof}
	\begin{satz}\label{Theorem:Continuity of top. K-theory}
		Let $G$ be an ample groupoid and $(A_n,\varphi_n)_n$ an inductive sequence of $G$-algebras. If we let $A=\lim A_n$, then the maps $\psi_{n,*}:\K_*^{\mathrm{top}}(G;A_n)\rightarrow \K_*^{\mathrm{top}}(G;A)$ induced by the canonical maps $\psi_n:A_n\rightarrow A$, give rise to an isomorphism
		$$\lim_{n\rightarrow\infty}\K_*^{\mathrm{top}}(G;A_n)\cong \K_*^{\mathrm{top}}(G;A).$$
	\end{satz}
	\begin{proof}
		Let $\psi^*:\lim_{n\rightarrow\infty}\K_*^{\mathrm{top}}(G;A_n)\rightarrow \K_*^{\mathrm{top}}(G;A)$ be the homomorphism induced by the morphisms $\psi_n: A_n\rightarrow A$. Our aim is to show that $\psi^*$ is an isomorphism. For every proper $G$-space $X$ let $$\psi_X^*: \lim_{n\rightarrow\infty}\mathrm{KK}^G_*(C_0(X),A_n)\rightarrow \mathrm{KK}_*^G(C_0(X),A)$$ be the morphism induced by $\psi_n$ at the level of $X$. Now the structure maps for taking the limit over $X$ are given by left Kasparov products, whereas the structure maps for taking the limit over the $A_n$ is given by right Kasparov products.
		Since the Kasparov product is associative, the limits can be permuted and we get $$\lim_{n\rightarrow\infty}\K_*^{\mathrm{top}}(G;A_n)\cong \lim_X\left( \lim_n \mathrm{KK}^G_*(C_0(X),A_n)\right).$$
		The map $\psi^*$ can then be computed via the maps $\psi_X^*$ by
		$$\lim_X\left( \lim_n \mathrm{KK}^G_*(C_0(X),A_n)\right)\rightarrow \lim_X \mathrm{KK}^G_*(C_0(X),A).$$
		We define a contravariant functor $$\mathcal{F}_H^*(C_0(X)):=\lim_n \mathrm{KK}_*^H(C_0(X),A_{n\mid H}).$$ Then $\mathcal{F}$ is a Going-Down functor. Let $\mathcal{G}$ denote the Going-Down functor $C_0(X)\mapsto \mathrm{KK}^H_*(C_0(X),A_{\mid H})$ from Example \ref{Example:Going-Down functor}. Then the maps $\psi_X$ define a Going-Down transformation $\Psi:\mathcal{F}\rightarrow\mathcal{G}$, such that $\Psi^*(G)=\psi^*$. By Theorem \ref{Theorem:Going-Down Theorem} it is hence enough to prove, that
		$$\lim\limits_n \mathrm{KK}^H_*(C(H^{(0)}),A_{n\mid H})\rightarrow \mathrm{KK}^H_*(C(H^{(0)}),A_{\mid H})$$
		is an isomorphism for all compact open subgroupoids $H$ in $G$.
		For every $n\in\NN$ we have a commutative diagram
		\begin{center}
			\begin{tikzpicture}[description/.style={fill=white,inner sep=2pt}]
			\matrix (m) [matrix of math nodes, row sep=3em,
			column sep=2.5em, text height=1.5ex, text depth=0.25ex]
			{ \mathrm{KK}^H_*(C(H^{(0)}),A_{n\mid H}) & \mathrm{KK}^H_*(C(H^{(0)}),A_{\mid H}) \\
				\K_*(A_{n\mid H}\rtimes H) & \K_*(A_{\mid H}\rtimes H) \\
			};
			\path[->,font=\scriptsize]
			(m-1-1) edge node[auto] {$ (\psi_n)_* $} (m-1-2)
			(m-1-2) edge node[auto] {$ \mu $} (m-2-2)
			(m-2-1) edge node[auto] { $ (\psi_n\rtimes H)_* $ } (m-2-2)
			(m-1-1) edge node[auto] { $ \mu_n $ } (m-2-1)
			;
			\end{tikzpicture},
		\end{center}
		where $\mu$ and $\mu_n$ are the isomorphisms coming from the groupoid version of the Green-Julg theorem (see \cite[Proposition~6.25]{Tu99}). By commutativity of the above diagrams it is hence enough to prove, that the maps $(\psi_n\rtimes H)_*$ induce an isomorphism
		$$\lim\limits_n\K_*(A_{n\mid H}\rtimes H)\rightarrow\K_*(A_{\mid H}\rtimes H).$$
		Using the continuity of $\K$-theory, the result follows from Lemma \ref{Lemma:Proper Groupoids and inductive limits}.
	\end{proof}
	As an immediate consequence, we get the following.
	\begin{kor}\label{Cor:InductiveLimit}
		Let $G$ be an ample groupoid and $(A_n,\varphi_n)_n$ an inductive sequence of $G$-algebras with $A=\lim_{n\rightarrow\infty}A_n$. Suppose $G$ satisfies the Baum-Connes conjecture with coefficients in $A_n$ for all $n\in\NN$. Assume further, that $G$ is exact, or that all the connecting homomorphisms $\varphi_n$ are injective. Then $G$ satisfies the Baum-Connes conjecture with coefficients in $A$.
	\end{kor}
	
\section{A Mixed Künneth Formula}\label{Section:MixedKunneth}
In this section we study the $\K$-theory of tensor products by crossed products with ample groupoids in analogy with the results from \cite{CEO}. The main tool is a mixed Künneth formula involving the topological K-theory of the groupoid in question. Under the assumtion that $G$ satisfies the Baum-Connes conjecture with coefficients, one can relate this mixed Künneth formula to the usual Künneth formula for the crossed product.
	Let us recall the usual Künneth formula: We say that a $\mathrm{C}^*$-algebra $A$ satisfies the \textit{Künneth formula} if for all $\mathrm{C}^*$-algebras $B$, there is a canonical short exact sequence
	$$0\longrightarrow \K_*(A)\otimes \K_*(B)\stackrel{\alpha}{\longrightarrow}\K_*(A\otimes B)\stackrel{\beta}{\longrightarrow}\mathrm{Tor}(\K_*(A),\K_*(B))\longrightarrow 0.$$
	The map $\alpha:\K_*(A)\otimes \K_*(B)\rightarrow \K_*(A\otimes B)$ in the above sequence can be obtained using the Kasparov product as the composition
	\begin{center}
		\begin{tikzpicture}[description/.style={fill=white,inner sep=2pt}]
		\matrix (m) [matrix of math nodes, row sep=3em,
		column sep=2em, text height=1.5ex, text depth=0.25ex]
		{ \mathrm{KK}(\CC,A)\otimes \mathrm{KK}_*(\CC,B) &  \mathrm{KK}_*(\CC,A)\otimes \mathrm{KK}_*(A,A\otimes B)\\
			&  \mathrm{KK}_*(\CC,A\otimes B) \\
		};
		\path[->,font=\scriptsize]
		(m-1-1) edge node[auto] {$\id\otimes\sigma_A $} (m-1-2)
		(m-1-2) edge node[auto] {$ \otimes_A $} (m-2-2)
		(m-1-1) edge node[auto] {$ \alpha  $} (m-2-2)
		
		;
		\end{tikzpicture}
	\end{center}
	where $\sigma_A:\KK_*(\CC,B)\rightarrow \KK_*(A,A\otimes B)$ is Kasparov's external tensor product in $\KK$-theory.
	We will also write $\alpha_{A,B}$ instead of $\alpha$ when it is important to remember which $\mathrm{C}^*$-algebras $A$ and $B$ we are considering.
	The following result is shown in \cite[Proposition~4.2]{CEO} (extending earlier results by \cite{MR650021}):
	\begin{prop}
		Let $A$ be a separable $\mathrm{C}^*$-algebra. Then $A$ satisfies the Künneth formula if and only if $\alpha:\K_*(A)\otimes \K_*(B)\rightarrow \K_*(A\otimes B)$ is an isomorphism for all separable $\mathrm{C}^*$-algebras $B$ with $\K_*(B)$ free abelian.
	\end{prop}
	The authors in \cite{CEO} define the class $\mathcal{N}$ to be the class of all separable $\mathrm{C}^*$-algebras such that $\alpha:\K_*(A)\otimes \K_*(B)\rightarrow \K_*(A\otimes B)$ is an isomorphism for all separable $\mathrm{C}^*$-algebras $B$ with $\K_*(B)$ free abelian. It turns out that the class $\mathcal{N}$ is quite large and enjoys many nice permanence properties:
	\begin{enumerate}
		\item The class $\mathcal{N}$ contains the bootstrap class $\mathcal{B}$ (see \cite[Definition~22.3.4]{MR1656031}).
		\item If $A\in\mathcal{N}$ and $B$ is $\mathrm{KK}$-dominated by $A$ (see \cite[Definition~23.10.6]{MR1656031}), then $B\in\mathcal{N}$.
		\item If $0\rightarrow I\rightarrow A\rightarrow A/I\rightarrow 0$ is a semi-split short exact sequence of $\mathrm{C}^*$-algebras such that two of them are in $\mathcal{N}$, then so is the third.
		\item If $A,B\in\mathcal{N}$, then $A\otimes B\in\mathcal{N}$.
		\item If $A=\lim_i A_i$ is an inductive limit such that each $A_i\in\mathcal{N}$ and, such that all the structure maps are injective, then $A\in\mathcal{N}$.
	\end{enumerate}
	
	Our first goal is to replace $\K_*(A)$ by the topological $\K$-theory of a groupoid $G$ with coefficients in a suitable separable $G$-algebra $A$ and define an equivariant version of the map $\alpha$.
	Before we can get into it, we need some preliminary observations on minimal tensor products of $C_0(X)$-algebras:
	
	Recall, that for arbitrary $\mathrm{C}^*$-algebras $A$ and $B$, their minimal tensor product $A\otimes B$ sits as an essential ideal inside $M(A)\otimes M(B)$, and hence, using the universal property of the multiplier algebra, there exists a unique embedding $\iota: M(A)\otimes M(B)\hookrightarrow M(A\otimes B)$, satisfying $\iota(m\otimes n)(a\otimes b)=ma\otimes nb$ and $(a\otimes b)\iota(m\otimes n)=am\otimes bn$. In particular, we have $\iota(ZM(A)\otimes ZM(B))\subseteq ZM(A\otimes B)$. In what follows we will suppress $\iota$ in our notation and view $ZM(A)\otimes ZM(B)$ as a subalgebra of $ZM(A\otimes B)$:
	\begin{prop}\cite[Proposition~3.4]{MR3549520}
		Let $A$ be a $C_0(X)$-algebra with structure map $\Phi_X$ and $B$ a $C_0(Y)$-algebra with structure map $\Phi_Y$. Then $A\otimes B$ is a $C_0(X\times Y)$-algebra with respect to the map $\Phi_X\otimes \Phi_Y$. Moreover, the fibre over $(x,y)\in X\times Y$ is $$(A\otimes B)_{(x,y)}=(A\otimes B)/I_x\otimes B+A\otimes J_y,$$ where $I_x$ and $J_y$ are the ideals corresponding to the fibres $A_x$ and $B_y$ respectively.
	\end{prop}
	In many situations the fibres are much nicer to describe:
	\begin{prop}\label{Prop:FibresMinimalTensorProduct}
		Let $A$ be a $C_0(X)$-algebra, and $B$ be a $C_0(Y)$-algebra. If either $A$ or $B$ is separable and exact, then $$(A\otimes B)_{(x,y)}=A_x\otimes B_y.$$
	\end{prop}
	\begin{proof}
		This is a direct consequence of \cite[IV.3.4.22, Proposition~IV.3.4.23]{MR2188261}.
	\end{proof}
	Now let $A$ and $B$ be $C_0(X)$-algebras over the same space $X$, and let $\Delta:X\rightarrow X\times X$ be the diagonal inclusion. Then we define the minimal balanced tensor product $A\otimes_X B$ of $A$ and $B$ by $\Delta^*(A\otimes B)$. Thus, $A\otimes_X B$ is a $C_0(X)$-algebra by construction. Note, that $A\otimes_X B$ is canonically isomorphic the quotient of $A\otimes B$ by the ideal $\overline{C_0(X\times X\setminus \Delta(X))A\otimes B}$.
	It follows from Proposition \ref{Prop:FibresMinimalTensorProduct} above, that if either $A$ or $B$ is separable and exact, that for all $x\in X$ we have
	$$(A\otimes_X B)_x=A_x\otimes B_x.$$
	With this description of the fibres it is not so hard to see the following:
	\begin{lemma}
		Let $A$ and $B$ be $C_0(X)$-algebras and $f:Y\rightarrow X$ a continuous map. If either $A$ or $B$ is separable and exact, we have $f^*(A\otimes_X B)\cong f^*A\otimes_Y f^*B$.
	\end{lemma}
	\begin{proof}
		Consider the map $f\times f:Y\times Y\rightarrow X\times X$. We will first show, that $f^*A\otimes f^*B$ is canonically isomorphic to $(f\times f)^*(A\otimes B)$ as a $C_0(Y\times Y)$-algebra. Consider the map
		$$\Phi:f^*A\otimes f^*B\rightarrow (f\times f)^*(A\otimes B),$$ 
		which on an elementary tensor $\varphi\otimes \psi\in f^*A\otimes f^*B$ is defined by
		$$\Phi(\varphi\otimes \psi)(y,y')=\varphi(y)\otimes\psi(y')
		\in A_{f(y)}\otimes B_{f(y')}=(f\times f)^*(A\otimes B)_{(y,y')}.$$
		Note, that we use the assumption that either $A$ or $B$ is separable and exact here, to identify the fibres in the last equality. Since
		\begin{align*}\norm{\Phi(\varphi\otimes\psi)}&=\sup\limits_{(y,y')} \norm{\varphi(y)\otimes\psi(y')}\\
		&=\sup\limits_{(y,y')}\norm{\varphi(y)}\norm{\psi(y')}\\
		&\leq\norm{\varphi}\norm{\psi}\\
		&=\norm{\varphi\otimes\psi},
		\end{align*}
		the map $\Phi$ extends to a bounded $C_0(Y\times Y)$-linear $\ast$-homomorphism, which clearly induces an isomorphism on each fibre. Hence $\Phi$ is an isomorphism as desired.
		Observe, that we have $\Delta_X\circ f=(f\times f)\circ \Delta_Y$, where $\Delta_X$ and $\Delta_Y$ denote the diagonal inclusions respectively. Hence we have
		\begin{align*}
		f^*(A\otimes_X B)=(\Delta_X\circ f)^*(A\otimes B)&=((f\times f)\circ \Delta_Y)^*(A\otimes B)\\
		&=\Delta_Y^*((f\times f)^*(A\otimes B))\\
		& \cong\Delta_Y^*(f^*A\otimes f^*B)\\
		&=f^*A\otimes_Y f^*B.
		\end{align*}
	\end{proof}
	Suppose now, that $G$ is a locally compact Hausdorff groupoid with a Haar system. Suppose further, that $(A,G,\alpha)$ and $(B,G,\beta)$ are groupoid dynamical systems. With the above lemma at hand, it is now easy to define a diagonal action. Suppose that either $A$ or $B$ is separable and exact. Then we define the diagonal action of $G$ on $A\otimes_{G^{(0)}} B$ via the composition
	$$d^*(A\otimes_{G^{(0)}} B)\cong d^*A\otimes_G d^*B\stackrel{\alpha\otimes \beta}{\longrightarrow}r^*A\otimes_G r^*B\cong r^*(A\otimes_{G^{(0)}} B).$$
	Note, that if $(A,G,\alpha)$ is a groupoid dynamical system and $B$ is any $\mathrm{C}^*$-algebra, such that either $A$ or $B$ is separable and exact, then $(A\otimes B,G,\alpha\otimes id)$ is a groupoid dynamical system. The reduced crossed product is compatible with the minimal balanced tensor product in the following way:
	\begin{prop}\cite[Theorem~6.1]{MR3383622}
		There is a natural isomorphism
		$$\Psi:(A\otimes B)\rtimes_{\alpha\otimes id,r} G\rightarrow (A\rtimes_{\alpha,r}G)\otimes B.$$
	\end{prop}
	Before we can proceed, we also need the following:
	\begin{prop}
		Let $A,B$ and $D$ be separable $G$-algebras, such that $D$ is exact. Then there is a homomorphism
		$$\sigma_D:\KK^G(A,B)\rightarrow \KK^G(A\otimes_{G^{(0)}}D,B\otimes_{G^{(0)}}D),$$
		given by associating to an element $(E,\Phi,T)\in\mathbb{E}^G(A,B)$ the triple $(E\otimes_A A\otimes_{G^{(0)}} D, \Phi\otimes \id,T\otimes\id)$.
	\end{prop}
	
	Let us now return to the Künneth formula. Fix a second countable locally compact Hausdorff groupoid $G$. For ease of notation let us denote its unit space by $X$.
	Let $A$ be a separable exact $G$-algebra and $B$ any $\mathrm{C}^*$-algebra. We wish to define a map
	$$\alpha_G:\K^{\mathrm{top}}_*(G;A) \otimes \K_*(B)\rightarrow \K_*^{\mathrm{top}}(G;A\otimes B).$$
	Consider the trivial group denoted by $1$. Then the canonical groupoid homomorphism $G\rightarrow 1$ induces a homomorphism
	$$\KK_*(\CC,B)\rightarrow \KK^G_*(C_0(X),C_0(X,B)).$$
	Let $\varepsilon$ denote the composition:
	\begin{center}
		\begin{tikzpicture}[description/.style={fill=white,inner sep=2pt}]
		\matrix (m) [matrix of math nodes, row sep=3em,
		column sep=2.5em, text height=1.5ex, text depth=0.25ex]
		{ \K_*(B) & \KK_*(\CC,B) &  \KK^G_*(C_0(X),C_0(X,B))\\
			& &  \KK_*^G(A\otimes_{X}C_0(X),A\otimes_{X}C_0(X,B)) \\
		};
		\path[->,font=\scriptsize]
		(m-1-1) edge node[auto] {$ \cong $} (m-1-2)
		
		(m-1-2) edge node[auto] {$  $} (m-1-3)
		(m-1-3) edge node[auto] {$ \sigma_A $} (m-2-3)
		(m-1-1) edge node[auto] {$ \varepsilon  $} (m-2-3)
		
		;
		\end{tikzpicture}
	\end{center}
	Under the canonical identifications of $G$-algebras $A\otimes_X C_0(X)\cong A$ and $A\otimes_X C_0(X,B)\cong A\otimes B$ we will view $\varepsilon$ as a map
	$$\varepsilon:\K_*(B)\rightarrow \KK_*^G(A,A\otimes B).$$
	Now for any proper and $G$-compact $G$-space $Y\subseteq \mathcal{E}(G)$  we define a map $\alpha_Y$ as the composition
	\begin{center}
		\begin{tikzpicture}[description/.style={fill=white,inner sep=2pt}]
		\matrix (m) [matrix of math nodes, row sep=3em,
		column sep=2.5em, text height=1.5ex, text depth=0.25ex]
		{ \KK_*^G(C_0(Y),A)\otimes \K_*(B) &  \KK_*^G(C_0(Y),A)\otimes \KK_*^G(A,A\otimes B)\\
			&   \KK_*^G(C_0(Y),A\otimes B) \\
		};
		\path[->,font=\scriptsize]
		(m-1-1) edge node[auto] {$\id\otimes\varepsilon $} (m-1-2)
		(m-1-2) edge node[auto] {$ \otimes_A $} (m-2-2)
		(m-1-1) edge node[auto] {$ \alpha_Y  $} (m-2-2)
		
		;
		\end{tikzpicture}
	\end{center}
	Passing to the limit, the maps $\alpha_Y$ induce the desired map
	$$\alpha_G:\K_*^{\mathrm{top}}(G;A)\otimes \K_*(B)\rightarrow \K_*^{\mathrm{top}}(G;A\otimes B).$$
	\begin{defi}
		We denote by $\mathcal{N}_G$ the class of all separable exact $G$-algebras $A$ such that $\alpha_G$ is an isomorphism for all $B$ with $\K_*(B)$ free abelian.
	\end{defi}
	We will now show, that for a $G$-algebra $A$ to be in $\mathcal{N}_G$ also corresponds to satisfying a $G$-equivariant version of the Künneth formula:
	\begin{prop}\label{Prop:Equivariant Kunneth sequence}
		Let $A$ be a separable and exact $G$-algebra. Then $A\in \mathcal{N}_G$ if and only if for every $\mathrm{C}^*$-algebra $B$, there exists a canonical homomorphism $$\beta_G:\K_*^{\mathrm{top}}(G;A\otimes B)\rightarrow \mathrm{Tor}(\K_*^{\mathrm{top}}(G;A),\K_*(B))$$ such that the sequence
		
			$$0\rightarrow \K_*^{\mathrm{top}}(G;A)\otimes\K_*(B)\stackrel{\alpha_G}{\rightarrow}\K_*^{\mathrm{top}}(G;A\otimes B)\stackrel{\beta_G}{\rightarrow}\mathrm{Tor}(\K_*^{\mathrm{top}}(G;A),\K_*(B))\rightarrow 0$$
		is exact.
	\end{prop}
	\begin{proof}
		Let $\mathcal{S}$ denote the category of all separable $\mathrm{C}^*$-algebras with $\ast$-homo\-morphisms as morphisms, and let $\mathbf{Ab}$ denote the category of abelian groups. Consider the functor $F_*:\mathcal{S}\rightarrow\mathbf{Ab}$, given by $F_*(B)=\K_*^{\mathrm{top}}(G;A\otimes B)$ and $F_*(\Phi)=(\id\otimes \Phi)_*$ for a $\ast$-homomorphism $\Phi:B_1\rightarrow B_2$. We will show that $F_*$ is a Künneth functor in the sense of \cite[Definition~3.1]{CEO}, provided that $A\in\mathcal{N}_G$. It is clear, that $F_*$ is stable and homotopy invariant, since the topological $\K$-theory has these properties. To see (K2), combine \cite[Lemma~4.1]{CEO} with \cite[Proposition~5.6]{Tu98}. Item (K3) again follows from the corresponding property of topological K-theory and (K4) is precisely what it means for $A$ to be in the class $\mathcal{N}_G$. Hence an application of \cite[Theorem~3.3]{CEO} completes the proof.	
	\end{proof}
	The class $\mathcal{N}_G$ enjoys many stability properties similar to those of $\mathcal{N}$:
	\begin{lemma}\label{Lemma:Stability Properties of N_G}
		Let $G$ be a second countable, locally compact Hausdorff groupoid. Then the following hold:
		\begin{enumerate}
			\item If $A\in\mathcal{N}_G$ and $B$ is a separable exact $\mathrm{C}^*$-algebra, which is $\KK^G$-dominated by $A$ (i.e. there exist $x\in\KK^G(A,B)$ and $y\in \KK^G(B,A)$ such that $y\otimes x=1_B\in \KK^G(B,B)$), then $B\in\mathcal{N}_G$.
			\item If $0\rightarrow I\rightarrow A\rightarrow A/I\rightarrow 0$ is a semi-split short exact sequence of $G$-algebras such that two of them are in $\mathcal{N}_G$, then so is the third.
			\item If $A\in\mathcal{N}_G$ and $B\in\mathcal{N}$, then $A\otimes B\in\mathcal{N}_G$, where $A\otimes B$ is equipped with the action $\alpha\otimes\id$.
			\item If $(A_n,\varphi_n)_n$ is an inductive sequence of $G$-algebras with injective and $G$-equivariant connecting maps, such that $A_n\in\mathcal{N}_G$ for all $n\in\NN$, then $A\in\mathcal{N}_G$.
		\end{enumerate}
	\end{lemma}
	\begin{proof}
		For the proof of $(1)$ let $D$ be any $\mathrm{C}^*$-algebra with $\K_*(D)$ free abelian and consider the following commutative diagram:
		\begin{center}
			\begin{tikzpicture}[description/.style={fill=white,inner sep=2pt}]
			\matrix (m) [matrix of math nodes, row sep=3em,
			column sep=3em, text height=1.5ex, text depth=0.25ex]
			{ \K_*^{\mathrm{top}}(G;B)\otimes \K_*(D) & \K_*^{\mathrm{top}}(G;A)\otimes \K_*(D) & \K_*^{\mathrm{top}}(G;B)\otimes \K_*(D)\\
				\K_*^{\mathrm{top}}(G;B\otimes D)& \K_*^{\mathrm{top}}(G;A\otimes D)&\K_*^{\mathrm{top}}(G;B\otimes D) \\
			};
			\path[->,font=\scriptsize]
			(m-1-1) edge node[auto] {$ (\cdot\otimes y)\otimes \id $} (m-1-2)
			(m-1-2) edge node[auto] {$ (\cdot\otimes x)\otimes \id $} (m-1-3)
			(m-2-1) edge node[auto] {$ \otimes \sigma_D(y) $} (m-2-2)
			(m-2-2) edge node[auto] {$ \otimes \sigma_D(x) $} (m-2-3)
			(m-1-1) edge node[auto] {$ \alpha_G  $} (m-2-1)
			(m-1-2) edge node[auto] {$ \alpha_G  $} (m-2-2)
			(m-1-3) edge node[auto] {$ \alpha_G  $} (m-2-3)
			;
			\end{tikzpicture}
		\end{center}
		By assumption, the composition of the horizontal arrows are the identity maps in each row and the middle vertical map is an isomorphism. An easy diagram chase then shows, that the left (and right) vertical arrows must be isomorphisms as well.
		
		For the proof of $(2)$, we first note that exactness passes to ideals (see \cite[Theorem~IV.3.4.3]{MR2188261}), quotients by \cite[Corollary~IV.3.4.19]{MR2188261} and semi-split extensions (see \cite[Theorem~IV.3.4.20]{MR2188261}) by deep results of Kirchberg and Wassermann. By \cite[Lemma~4.1]{CEO} the sequence
		$0\rightarrow I\otimes B\rightarrow A\otimes B\rightarrow A/I\otimes B\rightarrow 0$ is a semi-split short exact sequence as well, and hence $(2)$ follows from an easy application of the Five Lemma.
		
		For $(3)$ let us first observe, that if $A$ and $B$ are separable and exact $\mathrm{C}^*$-algebras, then so is their minimal tensor product $A\otimes B$ by associativity of the minimal tensor product. Now suppose that $A\in\mathcal{N}_G$ and $B\in\mathcal{N}$. Let $D$ be any $\mathrm{C}^*$-algebra with $\K_*(D)$ free abelian. As in the proof of \cite[Lemma~4.4(iii)]{CEO} we can use this fact to make the canonical identification 
		$$\Tor(\K_*^{\mathrm{top}}(G;A),\K_*(B)\otimes \K_*(D))\cong \Tor(\K_*^{\mathrm{top}}(G;A),\K_*(B))\otimes \K_*(D).$$
		Now consider the following commutative diagram:
		\begin{center}
			\begin{tikzpicture}[description/.style={fill=white,inner sep=2pt}]
			\matrix (m) [matrix of math nodes, row sep=3em,
			column sep=3.0em, text height=1.5ex, text depth=0.25ex]
			{ 	0 & 0\\
				\K_*^{\mathrm{top}}(G;A)\otimes \K_*(B)\otimes \K_*(D) & \K_*^{\mathrm{top}}(G;A)\otimes \K_*(B\otimes D)\\
				\K_*^{\mathrm{top}}(G;A\otimes B)\otimes \K_*(D)& \K_*^{\mathrm{top}}(G;A\otimes B\otimes D) \\
				\Tor(\K_*^{\mathrm{top}}(G;A),\K_*(B)\otimes \K_*(D)) & \Tor(\K_*^{\mathrm{top}}(G;A), \K_*(B\otimes D))\\
				0&0\\
			};
			\path[->,font=\scriptsize]
			(m-2-1) edge node[auto] {$ \id\otimes\alpha $} (m-2-2)
			(m-3-1) edge node[auto] {$ \alpha_G $} (m-3-2)
			(m-4-1) edge node[auto] {$ \Tor(\id,\alpha) $} (m-4-2)
			(m-1-1) edge node[auto] {$   $} (m-2-1)
			(m-1-2) edge node[auto] {$   $} (m-2-2)
			(m-2-1) edge node[auto] {$   $} (m-3-1)
			(m-2-2) edge node[auto] {$   $} (m-3-2)
			(m-3-1) edge node[auto] {$   $} (m-4-1)
			(m-3-2) edge node[auto] {$   $} (m-4-2)
			(m-4-1) edge node[auto] {$   $} (m-5-1)
			(m-4-2) edge node[auto] {$   $} (m-5-2)
			;
			\end{tikzpicture}
		\end{center}
		Under the identification of the $\Tor$ groups mentioned above, the first column is the equivariant Künneth sequence for $(A,B)$ tensored with $\K_*(D)$. Thus, using our assumption, that $A\in\mathcal{N}_G$, it is exact by Proposition \ref{Prop:Equivariant Kunneth sequence}. Similarly, the second column is the equivariant Künneth sequence for $(A,B\otimes D)$, and hence exact, too.
		Finally, the top and bottom arrows are isomorphisms, since $B$ was assumed to be in $\mathcal{N}$. By the Five Lemma, the middle vertical map $\alpha_G$ must be an isomorphism as well.
		
		Finally, for item $(4)$ note, that separability clearly passes to sequential inductive limits and exactness passes to inductive limits with injective connecting maps (see \cite[Proposition~IV.3.4.4]{MR2188261}). Hence the result follows from Theorem \ref{Theorem:Continuity of top. K-theory}.
	\end{proof}
	Using the Baum-Connes assembly map we can relate the map $\alpha_G$ to the map $\alpha$ for the crossed product as follows:
	\begin{prop}
		Let $A$ be a separable exact $G$-algebra and $B$ be any $\mathrm{C}^*$-algebra. Then the diagram
		\begin{center}
			\begin{tikzpicture}[description/.style={fill=white,inner sep=2pt}]
			\matrix (m) [matrix of math nodes, row sep=3em,
			column sep=2.5em, text height=1.5ex, text depth=0.25ex]
			{ \K_*^{\mathrm{top}}(G;A)\otimes \K_*(B) & \K_*(A\rtimes_r G)\otimes \K_*(B)\\
				\K_*^{\mathrm{top}}(G;A\otimes B)& \K_*((A\otimes B)\rtimes_r G) \\
			};
			\path[->,font=\scriptsize]
			(m-1-1) edge node[auto] {$ \mu_A\otimes \id $} (m-1-2)
			(m-2-1) edge node[auto] {$ \mu_{A\otimes B} $} (m-2-2)
			(m-1-1) edge node[auto] {$ \alpha_G  $} (m-2-1)
			(m-1-2) edge node[auto] {$ \alpha  $} (m-2-2)
			;
			\end{tikzpicture}
		\end{center}
		commutes. In particular, if $\mu_{A\otimes B}$ is an isomorphism for all $\mathrm{C}^*$-algebras $B$, then $A\in\mathcal{N}_G$ if and only if $A\rtimes_r G\in\mathcal{N}$.
	\end{prop}
	\begin{proof}
		First, note that for all $x\in \K_*(B)$ we have $j_G(\varepsilon(x))=\sigma_{A\rtimes_r G}(x)$, where $j_G:\KK^G(A,A\otimes B)\rightarrow \KK(A\rtimes_r G,(A\otimes B)\rtimes_r G)$ is the descent map (see \cite[Proposition~7.2.1]{LeGall}).
		Using this, we can easily check commutativity of the above diagram on the level of each $G$-compact subspace $Y\subseteq \mathcal{E}(G)$ as follows: For $y\in \KK^G_*(C_0(Y),A)$ and $x\in\K_*(B)$ we compute
		\begin{align*}
		\mu_{Y,A\otimes B}(\alpha_Y(y\otimes x)) & = [p_Y]\otimes_{C_0(Y)\rtimes_r G} j_G(\alpha_Y(y\otimes x))\\
		& = [p_Y]\otimes_{C_0(Y)\rtimes_r G} j_G(y\otimes_A \varepsilon(x))\\
		& = [p_Y]\otimes_{C_0(Y)\rtimes_r G} (j_G(y)\otimes_{A\rtimes_r G} \sigma_{A\rtimes_r G}(x))\\
		& = \mu_{Y,A}(y)\otimes \sigma_{A\rtimes_r G}(x)\\
		& = \alpha(\mu_{Y,A}(y)\otimes x).
		\end{align*}
		The second statement then follows directly from the commutativity of the diagram.
	\end{proof}
	We are now ready for the main result of this section, and this is the place where our techniques require us to restrict ourselves to ample groupoids.
	\begin{satz}\label{Theorem:Kunneth}
		Let $G$ be a second countable ample groupoid and $A$ a separable and exact $G$-algebra. Suppose that $A_{\mid K}\rtimes K\in\mathcal{N}$ for all compact open subgroupoids $K\subseteq G$. Then $A\in\mathcal{N}_G$.
	\end{satz}
	\begin{proof}
		Let $B$ be a fixed $\mathrm{C}^*$-algebra with $\K_*(B)$ free abelian. For each $H\in \mathcal{S}(G)$ define contravariant functors
		$\mathcal{F}_H:\mathcal{C}(H)\rightarrow\mathbf{Ab}$ and $\mathcal{G}_H:\mathcal{C}(H)\rightarrow\mathbf{Ab}$ by
		$$\mathcal{F}_H(C_0(Y)):=\KK^H_*(C_0(Y),A)\otimes \K_*(B),$$
		$$\mathcal{G}_H(C_0(Y)):=\KK^H_*(C_0(Y),A\otimes B).$$
		Both $(\mathcal{F}_H)_{H\in\mathcal{S}(G)}$ and $(\mathcal{G}_H)_{H\in\mathcal{S}(G)}$ define Going-Down functors in the sense of Definition \ref{Def:GDfunctor}.
		
		Moreover, for each $H\in \mathcal{S}(G)$ and every proper $H$-space $Y$ the maps $\alpha_Y$ determine natural transformations $\Lambda_H:\mathcal{F}_H\rightarrow\mathcal{G}_H$, which form a Going-Down transformation $\Lambda$.
		Our assumptions then translate to the fact that $\Lambda_K:\mathcal{F}_K(C(K^{(0)}))\rightarrow \mathcal{G}_K(C(K^{(0)}))$ is an isomorphism for every compact open subgroupoid $K$ of $G$. Hence, by Theorem \ref{Theorem:Going-Down Theorem} the result follows.
	\end{proof}
	The following corollary gives many examples, when the hypothesis of Theorem \ref{Theorem:Kunneth} are satisfied and thus provides many examples of $G$-algebras in class $\mathcal{N}_G$.
	\begin{kor}\label{Cor:Kunneth}
		Let $G$ be a second countable ample groupoid and $A$ be a separable exact $G$-algebra, such that $A_u$ is type I for all $u\in G^{(0)}$. Then $A\in\mathcal{N}_G$.
	\end{kor}
	\begin{proof}
		It follows from \cite[Proposition~10.3]{Tu98}, that $A_{\mid K}\rtimes K$ is a type I $\mathrm{C}^*$-algebra for all compact subgroupoids $K\subseteq G$, and hence it is contained in the bootstrap class $\mathcal{B}\subseteq \mathcal{N}$. The result then follows from Theorem \ref{Theorem:Kunneth}.
	\end{proof}
	Let us now point out the connections between the mixed Künneth formula and the Baum-Connes conjecture and state our main results concerning the Künneth formula for crossed products:
	\begin{prop}\label{Prop:BCandKunneth}
		Let $G$ be a second countable, locally compact Hausdorff groupoid with Haar system and $A\in\mathcal{N}_G$. Consider the following properties:
		\begin{enumerate}
			\item $G$ satisfies the Baum-Connes conjecture with coefficients in $A\otimes B$ for all separable $\mathrm{C}^*$-algebras $B$ (with respect to the trivial action on the second factor).
			\item $A\rtimes_r G\in\mathcal{N}$.
		\end{enumerate}
		Then $(1)$ implies $(2)$ and the converse holds, provided that $G$ satisfies the Baum-Connes conjecture with coefficients in $A$.
	\end{prop}
	\begin{proof}
		Consider the commutative diagram
		\begin{center}
			\begin{tikzpicture}[description/.style={fill=white,inner sep=2pt}]
			\matrix (m) [matrix of math nodes, row sep=3em,
			column sep=3em, text height=1.5ex, text depth=0.25ex]
			{ 	0 & 0\\
				\K_*^{\mathrm{top}}(G;A)\otimes \K_*(B) & \K_*(A\rtimes_r G)\otimes \K_*(B)\\
				\K_*^{\mathrm{top}}(G;A\otimes B)& \K_*((A\rtimes_r G)\otimes B) \\
				\Tor(\K_*^{\mathrm{top}}(G;A),\K_*(B)) & \Tor(\K_*(A\rtimes_r G), \K_*(B))\\
				0&0\\
			};
			\path[->,font=\scriptsize]
			(m-2-1) edge node[auto] {$ \mu_A\otimes\id $} (m-2-2)
			(m-3-1) edge node[auto] {$ \mu_{A\otimes B} $} (m-3-2)
			(m-4-1) edge node[auto] {$ \Tor(\mu_A,\id) $} (m-4-2)
			(m-1-1) edge node[auto] {$   $} (m-2-1)
			(m-1-2) edge node[auto] {$   $} (m-2-2)
			(m-2-1) edge node[auto] {$ \alpha_G  $} (m-3-1)
			(m-2-2) edge node[auto] {$  \alpha $} (m-3-2)
			(m-3-1) edge node[auto] {$  \beta_G $} (m-4-1)
			(m-3-2) edge node[auto] {$ \beta  $} (m-4-2)
			(m-4-1) edge node[auto] {$   $} (m-5-1)
			(m-4-2) edge node[auto] {$   $} (m-5-2)
			;
			\end{tikzpicture}
		\end{center}
		Since $A\in\mathcal{N}_G$ the left column is exact by Proposition \ref{Prop:Equivariant Kunneth sequence}. Now in the situation of $(1)$, all the horizontal arrows are isomorphisms. Consequently, the right column is also exact, which establishes $(2)$.
		If conversely $A\rtimes_r G\in\mathcal{N}$ and moreover $G$ satisfies the Baum-Connes conjecture with coefficients in $A$, then both columns in the above diagram are exact by Proposition \ref{Prop:Equivariant Kunneth sequence} and \cite[Proposition~4.2]{CEO} respectively. Moreover, the top and bottom horizontal maps are isomorphisms and an application of the Five Lemma completes the proof.
		
	\end{proof}
	Combining Theorem \ref{Theorem:Kunneth} and the preceding proposition we have the following corollary.
	\begin{kor}\label{Corollary:Kunneth}
		Let $G$ be a second countable ample groupoid and $A$ a separable exact $G$-algebra. Suppose the following hold:
		\begin{enumerate}
			\item $A_{\mid K}\rtimes K\in\mathcal{N}$ for all compact open subgroupoids $K\subseteq G$.
			\item $G$ satisfies the Baum-Connes conjecture with coefficients in $A\otimes B$ for all separable $\mathrm{C}^*$-algebras $B$ (with respect to the trivial action on the second factor).
		\end{enumerate}
		Then $A\rtimes_r G\in\mathcal{N}$.
				
		In particular we have $C_r^*(G)\in\mathcal{N}$, provided that $G$ satisfies the Baum-Connes conjecture with coefficients in $C_0(G^{(0)},B)$ for all separable $\mathrm{C}^*$-algebras $B$ (equipped with the trivial action).
	\end{kor}
	
	Using the above results we can also treat the case of twisted groupoid $\mathrm{C}^*$-algebras. Recall that a \textit{twist} $\Sigma$ over $G$ is a central groupoid extension
	$$G^{(0)}\times\TT\stackrel{i}{\longrightarrow} \Sigma\stackrel{j}{\longrightarrow} G,$$
	by which we mean:
	\begin{enumerate}
		\item The map $i$ is a homeomorphism onto $j^{-1}(G^{(0)})\subseteq \Sigma$,
		\item the map $j$ is a continuous and open surjection, and
		\item the extension is central meaning that $i(r(\sigma),z)\sigma=\sigma i(d(\sigma),z)$ for all $\sigma\in \Sigma$ and $z\in\TT$.
	\end{enumerate}
	Given such a twist $\Sigma$ over $G$ one can define a twisted version of the reduced $\mathrm{C}^*$-algebra of $G$, denoted by $C_r^*(G;\Sigma)$ (see \cite{MR1174207} for the details of this construction).
	\begin{kor}
		Let $G$ be a second countable ample groupoid and $\Sigma$ a twist over $G$. If $G$ satisfies the Baum-Connes conjecture with coefficients, then $C_r^*(G;\Sigma)\in\mathcal{N}$.
	\end{kor}
	\begin{proof}
		It was shown in \cite[Proposition~5.1]{VanErp} that there exists a Hilbert $C_0(G^{(0)})$-module $H$ and an action of $G$ on $K(H)$ such that $C_r^*(G;\Sigma)\sim_{Mor} K(H)\rtimes_r G$. Since each fibre $K(H)_u$ is canonically identified with the usual compact operators $K(H_u)$ on the Hilbert space $H_u$, which is a type I algebra, we have $K(H)\in\mathcal{N}_G$ by Corollary \ref{Cor:Kunneth} and hence $K(H)\rtimes_r G\in\mathcal{N}$ by Proposition \ref{Prop:BCandKunneth}. Since the class $\mathcal{N}$ is stable under Morita-equivalence, the result follows.
	\end{proof}

		Finally, we would like to enlarge the class of groupoids that our results can cover beyond the ample case. To this end we study the behaviour of the class $\mathcal{N}_G$ as $G$ varies.
		
		Let $X$ be a $G$-$H$-equivalence. Le Gall showed in \cite{LeGall,LeGall99} how to construct from an $H$-algebra $A$ a $G$-algebra $X^*A$ and a natural isomorphism  $X^*:\KK^H(A,B)\rightarrow \KK^G(X^*A,X^*B)$.
		\begin{prop}\label{Prop:Equivalence}
			Let $G$ and $H$ be second countable locally compact Hausdorff groupoids. Suppose that $X$ is a $G$-$H$-equivalence. Then $A\in \mathcal{N}_H$ if and only if $X^*A\in \mathcal{N}_G$.
		\end{prop}
		\begin{proof}
			Let $A$ be a separable exact $H$-algebra and $B$ be a separable $\mathrm{C}^*$-algebra with $\K_*(B)$ free abelian. It is an easy exercise to verify that, since $H$ acts trivially on $B$, we have an identification $X^*(A\otimes B)\cong X^*(A)\otimes B$.
			We only need to realize that the following diagram commutes, due to the compatibility of the map $X^*$ with the Kasparov product (see \cite[Theorème~6.2.2]{LeGall}).
			\begin{center}
				\begin{tikzpicture}[description/.style={fill=white,inner sep=2pt}]
				\matrix (m) [matrix of math nodes, row sep=3em,
				column sep=2.5em, text height=1.5ex, text depth=0.25ex]
				{ \K_*^{\mathrm{top}}(H;A)\otimes \K_*(B) & \K_*^{\mathrm{top}}(H;A\otimes B)\\
					\K_*^{\mathrm{top}}(G;X^*A)\otimes \K_*(B) & \K_*^{\mathrm{top}}(G;X^*A\otimes B)\\
				};
				\path[->,font=\scriptsize]
				(m-1-1) edge node[auto] {$ \alpha_H $} (m-1-2)
				(m-2-1) edge node[auto] {$ \alpha_G $} (m-2-2)
				(m-1-1) edge node[auto] {$ X^*\otimes \id  $} (m-2-1)
				(m-1-2) edge node[auto] {$ X^*  $} (m-2-2)
				;
				\end{tikzpicture}
			\end{center}
		\end{proof}
		
		\begin{prop}\label{Prop:TransformationGroupoids}
			Let $G$ be a second countable étale groupoid. Suppose $G$ acts on a (second countable) locally compact Hausdorff space $X$. Let $A$ be a separable exact $G\ltimes X$-algebra.
			Then $A\in \mathcal{N}_{G\ltimes X}$ if and only if $A\in \mathcal{N}_G$.
		\end{prop}
		\begin{proof}
			Let $A$ be a separable exact $G\ltimes X$-algebra and $B$ be a separable $\mathrm{C}^*$-algebra with $\K_*(B)$ free abelian. Then we have a commutative diagram
			\begin{center}
				\begin{tikzpicture}[description/.style={fill=white,inner sep=2pt}]
				\matrix (m) [matrix of math nodes, row sep=3em,
				column sep=2.5em, text height=1.5ex, text depth=0.25ex]
				{ \K_*^{\mathrm{top}}(G\ltimes X;A)\otimes \K_*(B) & \K_*^{\mathrm{top}}(G\ltimes X;A\otimes B)\\
					\K_*^{\mathrm{top}}(G;A)\otimes \K_*(B) & \K_*^{\mathrm{top}}(G;A\otimes B)\\
				};
				\path[->,font=\scriptsize]
				(m-1-1) edge node[auto] {$ \alpha_{G\ltimes X} $} (m-1-2)
				(m-2-1) edge node[auto] {$ \alpha_G $} (m-2-2)
				(m-1-1) edge node[auto] {$ F_A\otimes \id  $} (m-2-1)
				(m-1-2) edge node[auto] {$ F_{A\otimes B}  $} (m-2-2)
				;
				\end{tikzpicture}
			\end{center}
			where the vertical maps are given by the forgetful maps. Since $G$ is étale we know from \cite[Theorem~3.8]{Tu12} that the forgetful maps are isomorphisms. Hence the result follows.
		\end{proof}
			
		We will now combine these two cases to get a useful and checkable criterion. Recall, that a continuous groupoid homomorphism $\rho:G\rightarrow H$ is called \textit{faithful}, if the map $g\mapsto (r(g),\rho(g),s(g))$ is injective, and \textit{locally proper}, if the restriction of $\rho$ to the reduction $G_{\mid K}$ is proper for every compact subset $K\subseteq G^{(0)}$.
		\begin{prop}\label{Prop:LocallyProper} Let $G$ be a second countable, locally compact Hausdorff grouopid. Suppose $G$ admits a faithful, locally proper homomorphism $\rho:G\rightarrow H$ into a second countable étale groupoid $H$. Then there exists an $H$-space $Y$ and a $H\ltimes Y$-$G$ equivalence $X$. Moreover, if $A$ is a separable exact $G$-algebra, then $A\in\mathcal{N}_G$ if and only if $X^*A\in \mathcal{N}_{H}$.
		\end{prop}
		\begin{proof}
			The first part was proved in \cite{Delaroche} (see also \cite[Theorem~1.8]{MR1900993} for the case that $H$ is a group). The remaining part follows from combining Propositions \ref{Prop:Equivalence} and \ref{Prop:TransformationGroupoids}.
		\end{proof}
		
		So in particular if $\rho:G\rightarrow H$ is a faithful, locally proper homomorphism into an ample groupoid $H$, one can combine the preceding Proposition with Theorem \ref{Theorem:Kunneth} to obtain a criterion to determine whether a given $G$-algebra $A$ is in $\mathcal{N}_G$, even if $G$ is not ample.		
		The following special case of Proposition \ref{Prop:LocallyProper} is worth mentioning:
		\begin{kor}If $H\subseteq G$ is a closed subgroupoid, then $A\in \mathcal{N}_H$ if and only if $Ind_H^{G_{H^{(0)}}} A\in \mathcal{N}_G$.
		\end{kor}
		\begin{proof}
			The inclusion map $H\rightarrow G$ is faithful and locally proper. Finally, $G_{H^{(0)}}$ implements the equivalence between $G\ltimes G_{H^{(0)}}/H$ and $H$ and $(G_{H^{(0)}})^*(A)=Ind_H^{G_{H^{(0)}}} A$ (see \cite[Remark~3.18]{1806.00391}), so the result follows from the previous proposition.
		\end{proof}

\section{Applications}\label{Section:Applications}
A natural question is to find examples of $\mathrm{C}^*$-algebras for which Corollary \ref{Corollary:Kunneth} ensures the Künneth formula, without this being a consequence of previous known results \cite{RosenbergKunneth,CEO,OY4}. This is actually not an easy question (in the second author's opinion). The first class of examples we give are actually also consequences of a combination of previous results, but they are obtained here in a straightforward way using the Going-Down principle for groupoids. The authors thought they were interesting in themselves, and allow to illustrate the power of our results (or at least to compare it to other methods). Genuine new applications will be provided by uniform Roe algebras of spaces which admit a coarse embedding into a Hilbert space, and the uniform maximal Roe algebra of spaces which admit a fibred coarse embedding. This has, to our knowledge, not appeared in the literature before, and is an easy application of results. 

\subsection{Some examples}

Let $\Gamma$ an infinite hyperbolic property (T) group, and $\Omega$ a second countable locally compact Hausdorff space equipped with an action of $\Gamma$ by homeomorphisms. Then the action groupoid $\Gamma\ltimes\Omega$ is étale. 
By Lafforgue's work in \cite{lafforgue2012conjecture}, $\Gamma$ satisfies the Baum-Connes conjecture with coefficients, so \cite[Corollary~0.2]{MR1966758}
ensures that so does $\Gamma\ltimes\Omega$. By Proposition \ref{Prop:TransformationGroupoids} we have $A\in \mathcal{N}_{\Gamma\ltimes\Omega}$ if and only if $A\in \mathcal{N}_\Gamma$. Hence an application of Theorem \ref{Theorem:Kunneth} together with Proposition \ref{Prop:BCandKunneth} ensures the following:\\
\paragraph{\textbf{Application 1}} If $\Gamma$ is an infinite hyperbolic property (T) group, $\Omega$ a $\Gamma$-space, and $A$ is a $\Gamma\ltimes\Omega$-algebra such that $A\rtimes_r F$ satisfies the Künneth formula for every finite subgroup $F$ of $\Gamma$, then $A \rtimes_r (\Gamma\ltimes\Omega)$ satisfies the Künneth formula.

This result is not new and is a consequence of \cite{CEO} applied to $\Gamma$ once we realize that $A\rtimes_r (\Gamma\ltimes\Omega)\cong A\rtimes_r \Gamma$.
\\

Recall the following construction from \cite{MR1911663}. Let $\Gamma$ be a finitely generated residually finite group, and $\mathcal N = \lbrace N_i\rbrace_i$ a decreasing family of nested finite index normal subgroups, i.e. $N_{i+1} < N_i $, $[\Gamma, N_i]<\infty$ and $\bigcap N_i =\{e_\Gamma\}$. Following \cite{MR3549528}, we define the HLS groupoid (after Higson, Lafforgue, and Skandalis) associated to $(\Gamma,\mathcal N)$ as the bundle of groups over the one point compactification $\overline{\N}=\N\cup \{\infty\}$ as follows:
\begin{itemize}
\item[$\bullet$] if $n\in\N$, $\Gamma_n= \Gamma / N_n $,
\item[$\bullet$] $\Gamma_\infty=\Gamma$,
\item[$\bullet$] each fibre over $\NN$ is endowed with the discrete topology, and a neighbourhood basis of $(\infty,g)$ is given by   
\[ \mathcal V_{N} = \{(n,g_n) : n\geq N, \pi_n(g_n) = g\}.\]
\end{itemize}
This defines an ample groupoid $G_\mathcal{N}(\Gamma)$, and the exact sequence
\[ 0 \rightarrow \oplus \C[\Gamma_n] \rightarrow C_c(G_\mathcal{N}(\Gamma)) \rightarrow \C [\Gamma]  \rightarrow 0\]      
induces the following exact sequence of $C^*$-algebras
\[ 0 \rightarrow \oplus \C[\Gamma_n] \rightarrow C_r^*(G_\mathcal{N}(\Gamma)) \rightarrow C^*_{\mathcal N} (\Gamma)  \rightarrow 0,\]   
where $C^*_{\mathcal N}(\Gamma)$ is the completion of $\C[\Gamma]$ w.r.t. to the norm
\[||x||_{\mathcal N} = \sup_{N\in \mathcal N} ||\lambda_{N} (x)||\quad x\in \C[\Gamma] \]
induced by the quasi-regular representations $\lambda_{N} : C_{max}^*(\Gamma) \rightarrow \mathcal L(l^2(\Gamma/ N))$.  \\ 

Now this exact sequence intertwines the Baum-Connes assembly maps, and the Baum-Connes conjecture for $G_{\mathcal N}(\Gamma)$ is equivalent to $\mu_{\Gamma,\mathcal N}$ being an isomorphism. \\

\paragraph{\textbf{Application 2}} 
\begin{itemize}
\item[$\bullet$] If $\Gamma= \mathbb F_2$ and 
\[N_n = \bigcap \ker(\phi) \]
for $\phi$ running across all group homomorphisms from $\Gamma$ to a finite group of cardinality at most $n$, then $C_{\mathcal N}^*(\Gamma) \cong C_{max}^*(\Gamma)$ (see \cite[Lemma~2.8]{MR3549528}). Since $G_\mathcal{N}(\FF_2)$ is an ample groupoid satisfying the Baum-Connes conjecture we get that $C_r^*(G_\mathcal{N}(\FF_2))$ satisfies the Künneth formula by Corollary \ref{Corollary:Kunneth}. It is still a result that one can get using the fact that $\FF_2$ being a-T-menable, it is $\K$-amenable. Hence $C^*_{max}(\FF_2)$ and $C_r^*(\FF_2)$ are $\KK$-equivalent and in the class $\mathcal{N}$, so that $C_r^*(G_\mathcal{N}(\FF_2))$ also is by extension stability of $\mathcal{N}$. A remark of R. Willett is worth mentioning: $\mathbb F_2$ being the fundamental group of the wedge of two circles, it is $\KK$-equivalent to $C(\mathbb S^1 \wedge \mathbb S^1)$.

\item[$\bullet$] One can artificially try to get rid of bootstrapiness by spatially tensoring this exact sequence by $C_r^*(\Lambda)$ for a infinite hyperbolic property (T) group $\Lambda$. Since hyperbolic groups are exact we get the extension
\[ 0 \rightarrow \oplus \C[\Gamma_n] \otimes_{min} C_r^*(\Lambda) \rightarrow C_r^*(G_\mathcal{N}(\Gamma)\times \Lambda) \rightarrow C^*_{\mathcal N} (\Gamma)\otimes_{min} C_r^*(\Lambda)   \rightarrow 0.\]
Our methods apply to the groupoid $G_{\mathcal N}(\Gamma)\times\Lambda$, and imply that its reduced $C^*$-algebra satisfies the Künneth formula. But then again, one can deduce this from a previous result, namely the restriction principle for groups. Indeed, apply it to $\Lambda$ acting trivially on $C_r^*(G_\mathcal{N}(\Gamma))$.
\end{itemize} 


\subsection{Uniform and maximal Roe algebras}




Let $X$ be a discrete metric space with bounded geometry, which means that for every given $R>0$, there is a uniform bound on the cardinality on $R$-balls in $X$, i.e. 
\[\sup_{x\in X} |B(x,R)| < \infty \quad \forall R>0.\]
Motivated by index theory in the setting of non-compact Riemanian manifolds, J. Roe introduced a $C^*$-algebra $C^*(X)$, now called the Roe algebra of $X$. It comes in different versions, which are all completions of the $*$-algebra of locally compact finite propagation operators, of which we now recall the definition. If $H$ is an auxiliary separable Hilbert space and $R>0$, define 
\[ \C_R[X] = \{T\in\mathcal B(l^2(X)\otimes H)\text{ s.t. } T_{xy}\in \mathfrak K(H) \text{ and } T_{xy} = 0 \text{ if } d(x,y)>R \} \]
as a closed subspace of $\mathcal B(l^2(X)\otimes H)$. Then 
\[\C[X] = \cup_{R>0} \C_R[X]\]
is a $*$-algebra, and
\begin{itemize}
\item[$\bullet$] if $H=\C$, the completion of $\C[X]$ inside $\mathcal B(l^2(X))$ is called the \textit{uniform Roe algebra} $C_u^*(X)$ of $X$,
\item[$\bullet$] if $H=l^2(\N)$, the completion of $\C[X]$ inside $\mathcal B(l^2(X)\otimes H)$ is called the \textit{Roe algebra} $C^*(X)$  of $X$,
\item[$\bullet$] the envelopping $C^*$-algebra of $\C[X]$ is called the \textit{maximal Roe algebra} $C_{max}^*(X)$  of $X$.
\end{itemize}
In \cite{STY02}, G. Skandalis, J.-L. Tu and G. Yu define a locally compact Hausdorff ample groupoid $G(X)$, called the \textit{coarse groupoid} of $X$, such that the convolution algebra $C_c(G(X))$ is $*$-isomorphic to $\C[X]$, and 
\begin{itemize}
\item[$\bullet$] $C^*_r(G(X)) \cong C_u^*(X)$,
\item[$\bullet$] $l^\infty_X=l^\infty(X, \mathfrak K(H))$ is a $G(X)$-algebra and $C^*(X)\cong l^\infty_X \rtimes_r G(X)$,
\item[$\bullet$] $C_{max}^*(G(X)) \cong C_{max}^*(X)$.
\end{itemize}
An interesting feature of the coarse groupoid is that its dynamical properties reflect the coarse properties of $X$. Indeed,
\begin{itemize}
\item[$\bullet$] $X$ has property A iff $G(X)$ is amenable iff $G(X)$ is inner exact (see \cite[Theorem~5.3]{STY02} and \cite[Theorem~3.20]{bonicke_li_2018}),
\item[$\bullet$] $X$ coarsely embeds into a Hilbert space iff $G(X)$ is a-T-menable \cite[Theorem~5.4]{STY02},
\item[$\bullet$] $X$ admits a fibred coarse embedding into a Hilbert space iff $G(X)_{|\partial \beta X }$ is a-T-menable (see \cite{FinnSellFibred}). Here, $\partial \beta X$ is the complement of $X$ in $\beta X$, hence is a closed $G(X)$-invariant subset of $\beta X$.
\end{itemize}
Fibred coarse embeddings were introduced by Chen, Wang and Yu in \cite{ChenWangYu}. Admitting a fibred coarse embedding is weaker than admitting a coarse embedding: For instance any box space of an a-T-menable group admits a fibred coarse embedding, hence one gets examples of expanders which admit fibred coarse embeddings.\\

We want to study, when the different Roe-type algebras defined above satisfy the Künneth formula. First we focus on the uniform Roe-algebra. Suppose $X$ is a bounded geometry metric space, which admits a coarse embedding into a Hilbert space.
One would wish to apply the results of the present work directly, but the problem is that $G(X)$ is not second countable, so that the $\mathrm{C}^*$-algebras above are not separable. However, by \cite[Theorem~5.4]{STY02} there exists a second countable a-T-menable groupoid $G'$ such that $G(X)=G'\ltimes \beta X$. Moreover, we may write $\beta X$ as an inverse limit $\lim\limits_{\longleftarrow} Y_i$, such that:
\begin{itemize}
\item Each $Y_i$ is a metrizable quotient of $\beta X$.
\item The action of $G'$ on $\beta X$ factors through an action of $G'$ on $Y_i$, making the quotient map $G'$-equivariant.
\item For each $i\leq j$ the canonical map $Y_j\rightarrow Y_i$ is $G'$-equivariant and surjective.
\end{itemize}
Consequently, the coarse groupoid $G(X)$ can be written as a projective limit of second countable ample groupoids. 

\begin{satz}\label{Theorem:Uniform Roe-algebra}
		Let $X$ be a discrete metric space with bounded geometry, which admits a coarse embedding into a Hilbert space. Then $C_u^*(X)$ satisfies the Künneth formula.
	\end{satz}
	\begin{proof}
		With the notations above, we get an inductive system $(C(Y_i))_i$ of separable $G'$-algebras with injective and $G'$-equivariant connecting homomorphisms. Hence $(C(Y_i)\rtimes_r G')_i$ is an inductive system with injective connecting maps as well and we can apply Lemma \ref{Lemma:Proper Groupoids and inductive limits} to obtain
		$$C_u^*(X)=C_r^*(G(X))=C(\beta X)\rtimes_r G'=\lim\limits_{\longrightarrow} C(Y_i)\rtimes_r G'.$$
		Now $G'$ is a second countable ample groupoid and satisfies the Baum-Connes conjecture for all coefficients by \cite{Tu98}. If $K\subseteq G'$ is a compact open subgroupoid, then $C(Y_i)_{\mid K}\rtimes K=C_r^*(K\ltimes(Y_i)_{\mid K})$ is the $\mathrm{C}^*$-algebra of a compact groupoid, and hence contained in the class $\mathcal{N}$. By Corollary \ref{Corollary:Kunneth} it follows that $C(Y_i)\rtimes_r G'$ satisfies the Künneth formula for all $i$, i.e. for every $\mathrm{C}^*$-algebra $B$ we obtain canonical short exact sequences
		$$0\rightarrow \K_*(C(Y_i)\rtimes_r G')\otimes \K_*(B)\rightarrow \K_*((C(Y_i)\rtimes_r G')\otimes B)\rightarrow \Tor(\K_*(C(Y_i)\rtimes_r G'),\K_*(B))\rightarrow 0.$$
		Since the connecting maps of the directed system $(C(Y_i)\rtimes_r G')_i$ are all injective by construction, we can apply \cite[II.9.6.6]{MR2188261} to get $$\lim\limits_{\longrightarrow}(C(Y_i)\rtimes_r G')\otimes B=C_u^*(X)\otimes B.$$
		As the (algebraic) tensor product functor and the Tor functor commute with inductive limits, and $\mathrm{K}$-theory is continuous, we obtain a short exact sequence
		$$0\rightarrow \K_*(C_u^*(X))\otimes \K_*(B)\rightarrow \K_*(C_u^*(X)\otimes B)\rightarrow \Tor(\K_*(C_u^*(X)),\K_*(B))\rightarrow 0,$$
		in the limit, as desired.
	\end{proof}

	We can also blend Application 1 and the preceding result: If $\Gamma$ is a discrete group, then it has only one coarse class of left invariant metric, and we denote by $|\Gamma|$ the associated coarse metric space. It is shown in \cite{STY02} that the coarse groupoid is an action groupoid, more precisely \[G(|\Gamma|) \cong \beta |\Gamma| \rtimes \Gamma.\]
	This groupoid is ample and satisfies the Baum-Connes conjecture with coefficients, provided that $\Gamma$ is embeds coarsely into a Hilbert space (for example if $\Gamma$ is exact).
	
	\paragraph{\textbf{Application 3}} If $\Gamma$ is a hyperbolic group and $A$ is a $C^*$-algebra such that $A\rtimes_r F$ satisfies the Künneth formula for every finite subgroup $F$ of $\Gamma$, then $A\rtimes_r G(\betrag{\Gamma})$ satisfies the Künneth formula. In particular the uniform Roe algebra \[l^\infty(\Gamma)\rtimes_r \Gamma\cong C_u^* (\Gamma)\] satisfies the Künneth formula.

\begin{satz}\label{Theorem: Maximal Roe-algebra}
		Let $X$ be a discrete metric space with bounded geometry, which admits a fibred coarse embedding into a Hilbert space. Then $C_{max}^*(X)$ satisfies the Künneth formula.
\end{satz}
\begin{proof}
The closed saturated subset $\partial \beta X$ gives rise to the following exact sequence of $C^*$-algebras
\[0 \rightarrow C_{max}^*(X\times X) \rightarrow C^*_{max}(G(X)) \rightarrow C^*_{max}(G(X)_{|\partial \beta X}) \rightarrow 0.\] 
The groupoid $X\times X$ being proper, the $C^*$-algebra on the left side is of type I and satisfies the Künneth formula. By the above, 
\[G(X)_{|\partial \beta X} = \partial \beta X \rtimes G' =\lim\limits_{\longleftarrow} X_i\cap G', \]
where $X_i$ is the image of $\partial \beta X$ under the $G(X)$-equivariant quotient map $\beta X \rightarrow Y_i$.\\

By hypothesis, $X$ admits a fibred embedding into a Hilbert space, and hence $G(X)_{|\partial \beta X}$ is a-T-menable. The same argument as in the previous result ensures that $C^*_{r}(G(X)_{|\partial \beta X})$ satisfies the Künneth formula. But a-T-menable groupoids have $\K$-amenable $\mathrm{C}^*$-algebras by \cite{Tu98}. It follows that $C^*_{max}(G(X)_{|\partial \beta X})$ satisfies the Künneth formula. 
The Künneth formula is stable by extension, which concludes the proof.
\end{proof}

\section{Stability result}\label{Section:Stability}
We will end this paper with an extension of the main result. Indeed, the class of second countable ample groupoids can be used as a starting point of an inductively defined class of groupoids whose $\mathrm{C}^*$-algebras satisfy the Künneth formula. This class is directly inspired from the class of finite decomposition complexity for groups \cite{GTY}. \\

Next we state a strengthening of Corollary \ref{Corollary:Kunneth}, which is proved in the second author's thesis \cite{DellAieraThesis}. It says that the Künneth morphism $\alpha_{A\rtimes_r G,B}$ comes from a controlled morphism $\hat \alpha_{A\rtimes_r G,B}$, which is a quantitative isomorphism.

\begin{satz}[Theorem $5.2.13$ \cite{DellAieraThesis}]
Let $G$ be a ($\sigma$-compact) second countable ample groupoid and $A$ a separable and exact $G$-algebra. Suppose that 
\begin{itemize}
\item[$\bullet$] $G$ satisfies the Baum-Connes conjecture with coefficients in $A\otimes B $ for every separable trivial $G$-algebra $B$,
\item[$\bullet$] for every compact open subgroupoid $K$ of $G$, $A_{|K}\rtimes_r K \in \mathcal N$.
\end{itemize} 
Then $A\rtimes_r G$ satisfies the quantitative Künneth formula.
\end{satz}

Quantitative $\K$-theory was developed by H. Oyono-Oyono and G. Yu in \cite{OY2}, and its application to the Künneth formula in \cite{OY4}. The main topic of the first author's thesis \cite{DellAieraThesis} was a generalization of operator quantitative $\K$-theory, called controlled $\K$-theory, which allows to state that crossed products $A\rtimes_r G$ are $C^*$-algebras which are filtered by the set of symmetric compact subsets $K\subseteq G$. One can then study the controlled $\K$-theory group $\hat \K_*(A\rtimes_r G)$, which approximates $\K_*(A\rtimes_r G)$ in a precise sense. We refer the reader to \cite{DellAieraThesis} or \cite{dell2017controlled} for more details. The proof of the quantitative Künneth formula is essentially the same as the classical one. One just has to use the controlled version of every morphism involved, and has to keep track of the propagation at every step. The quantitative Künneth formula essentially means that the morphism $\alpha_{A\rtimes_r G, B}$ is induced by a controlled morphism $\hat\alpha_{A\rtimes_r G , B}$.\\

In \cite{OY4}, H. Oyono-Oyono and G. Yu introduced the class $C_{fand}$ of finite asymptotic nuclear dimensional $\mathrm{C}^*$-algebras, and show \cite[Proposition~5.6]{OY4} that every member of this class satisfies the Künneth formula. To define this class, we first need to recall what is a filtered $\mathrm{C}^*$-algebra, and a controlled Mayer-Vietoris pair.

\begin{defi}
A coarse structure is a poset $\mathcal E$ equipped with an abelian semi group structure such that, for any two elements $E,E'\in \mathcal E$, there exists an element $F\in \mathcal E$ such that $E\leq F$ and $E'\leq F$. A $C^*$-algebra $A$ is said to be $\mathcal E$-filtered if there exists a family $\{A_E \}_{E\in \mathcal E}$ of closed self-adjoint subspaces of $A$ such that:
\begin{itemize}
\item[$\bullet$] $A_E \subseteq A_{E'}$ if $E\leq E'$,
\item[$\bullet$] $A_E . A_{E'} \subseteq A_{EE'}$,
\item[$\bullet$] $\cup_{E\in \mathcal E} A_E$ is dense in $A$.
\end{itemize} 
If $A$ is unital, we impose that $1\in A_E$ for every $E\in \mathcal E$.
\end{defi} 

Examples of filtered $\mathrm{C}^*$-algebras include Roe algebras associated to proper metric spaces with bounded geometry, crossed-products of $\mathrm{C}^*$-algebras by action by automorphisms of \'etale groupoids or discrete quantum groups. See \cite[Chapter~3]{DellAieraThesis} or \cite{dell2017controlled} for details. Any sub-$\mathrm{C}^*$-algebra $B$ of $A$ is considered filtered by the family $\{B\cap A_E\}_{E\in \mathcal E}$. If $A$ and $A'$ are $\mathcal E$-filtered, then $A\cap A'$ is considered filtered by the family $\{A_E\cap A'_E\}_{E\in \mathcal E}$. \\

To a $\mathcal E$-filtered $\mathrm{C}^*$-algebra $A$, one can associate its controlled $\K$-theory groups $\hat \K_*(A)$ which is a family of groups 
\[\{ \K^{\varepsilon, E}_*(A) \}_{\varepsilon\in (0, \frac{1}{4}), E\in \mathcal E}\] 
satisfying nice compatibility conditions and approximating the $\K$-theory groups $\K_*(A)$. A controlled morphism $\hat\phi = \{\phi_{\varepsilon, E}\}$ is a family of morphisms
\[ \phi_{\varepsilon, E} : \K^{\varepsilon, E}_*(A) \rightarrow \K^{\alpha\varepsilon, h_\varepsilon .E}_*(B) \quad 
\forall \varepsilon \in (0, \frac{1}{4\alpha}), E\in \mathcal E \]
where $\alpha \leq 1$ is a fixed constant, and $h$ is a nondecreasing function. The point is that the way the propagation, i.e. the parameters are distorted, is uniform accross the family. The family must satisfy compatibility conditions we do not recall here in order to keep a reasonable length for the article (and the details of controlled $\K$-theory are not essential for the proof). Forgetting the propagation, any controlled morphism $\hat \phi$ induces a morphism in $\K$-theory  
\[ \phi : \K_*(A) \rightarrow \K_*(B) . \]
One of the interest of controlled $\K$-theory lies in its computability. For instance, Mayer-Vietoris type exact sequences occur even if the filtered $\mathrm{C}^*$-algebra is simple. More precisely, in \cite{OY4} is developed a notion of controlled Mayer-Vietoris pair, which we now recall.

\begin{defi}
Let $A$ be a $\mathcal E$-filtered $\mathrm{C}^*$-algebra, $c\geq 1$ and $F\in \mathcal E$. A $F$-controlled Mayer-Vietoris pair with coercivity $c$ is a quadruple $(V_0, V_1, A^{(0)}, A^{(1)})$:
\begin{itemize}
\item[$\bullet$] the $V_i$'s are closed subspaces of $A_F$,
\item[$\bullet$] $A^{(i)}$ is a $\mathrm{C}^*$-algebra containing \[ V_i + A_{F'} V_i + V_i A_{F'}  + A_{F'} V_i A_{F'}\]
with $F' = F^5$,
\item[$\bullet$] for every $E\leq F$, every $x\in M_n(A_E)$ can be written as a sum \[x=x_0+x_1\] where $x_i\in M_n(V_i \cap A_E)$ and $|| x_i|| \leq c||x||$,
\item[$\bullet$] for every $\varepsilon>0$, $E\leq F$ and every $\varepsilon$-close elements $x\in A_E^{(0)}$ and $y\in A_E^{(1)}$, i.e.
\[|| x-y || < \varepsilon,\]
there exists $z\in M_n( A_E^{(0)}\cap A_E^{(1)})$ such that \[ ||x-z|| < c\varepsilon \quad \text{and} \quad ||y-z|| < c\varepsilon .\]
\end{itemize}
If $\mathcal A$ and $\mathcal B$ are two families of $\mathcal E$-filtered $\mathrm{C}^*$-algebras, we say that $\mathcal A$ $2$-decomposes over $\mathcal B$ if there exists a constant $c\geq 1$ such that, for every $A\in\mathcal A$, and every $E\in \mathcal E$, there exists a controlled Mayer-Vietoris pair $(V_0, V_1, A^{(0)}, A^{(1)})$ with coercivity $c$ with $A^{(0)}$, $A^{(1)}$ and $A^{(0)} \cap A^{(1)}$ belonging to $\mathcal B$.
\end{defi} 

If in possession of a controlled Mayer-Vietoris pair $(V_0, V_1, A^{(0)}, A^{(1)})$ for a filtered $\mathrm{C}^*$-algebra $A$, \cite[Theorem~3.10]{OY4} allows to compute its controlled $\K$-theory in terms of the controlled $\K$-theory of the sub-$\mathrm{C}^*$-algebras $A_i$. See \cite{OY4,DellAieraThesis,dell2017controlled} for precise definitions about controlled morphisms and controlled exact sequences. 

\begin{satz}
For every $\mathcal E$-filtered $\mathrm{C}^*$-algebra $A$, $E\in \mathcal E$ and every $E$-controlled Mayer-Vietoris pair $(V_0, V_1, A^{(0)}, A^{(1)})$, there exists a controlled sequence
\[\begin{tikzcd}
 \hat \K_*( A^{(0)}\cap A^{(1)} ) \arrow{r} & \hat \K_*(A^{(0)}) \oplus \hat \K_*(A^{(1)}) \arrow{r} & \hat \K_*(A) \arrow{d} \\ 
 \hat \K_*(A) \arrow{u} & \hat \K_*(A^{(0)}) \oplus \hat \K_*(A^{(1)}) \arrow{l} & \hat \K_*( A^{(0)}\cap A^{(1)} ) \arrow{l}
\end{tikzcd}\]
which is controlled-exact up to order $E$.  
\end{satz}  

This result allows H. Oyono-Oyono and G. Yu to prove a permanence result \cite[Theorem~4.12]{OY4}.

\begin{satz}
Let $A$ be a $\mathcal E$-filtered $\mathrm{C}^*$-algebra. If for every $E\in \mathcal E$ there exists a $E$-controlled Mayer-Vietoris pair $(V_0, V_1, A^{(0)}, A^{(1)})$ such that $A^{(0)}$, $A^{(1)}$ and $A^{(0)} \cap A^{(1)}$ satisfy the quantitative Künneth formula then $A$ satisfies the quantitative Künneth formula.   
\end{satz}

Let $\mathcal E$ be a coarse structure. A $\mathcal E$-filtered $\mathrm{C}^*$-algebra $A$ is said to be \textit{locally bootstrap} if, for every $E\in \mathcal E$, there exists $F\in \mathcal E$ and a sub-$\mathrm{C}^*$-algebra $A^{(F)}$ of $A$, which is in the bootstrap class $\mathcal B$ and satisfies
\[A_E \subseteq A^{(F)}\subseteq A_F. \]
Notice the following property: a locally bootstrap $\mathrm{C}^*$-algebra is automatically in the bootstrap class. It is indeed an inductive limit of $\mathrm{C}^*$-algebras in the bootstrap class. Denote by $C_{fand}^{(0)}$ the class of locally bootstrap $\mathrm{C}^*$-algebras. Then, a $\mathrm{C}^*$-algebra $A$ belongs to the class $C^{(n+1)}_{fand}$ if it is $2$-decomposable over $C_{fand}^{(n)}$. \\

The asymptotic nuclear dimension of $A$ is the smaller $n$ such that $A$ belongs to $C^{(n)}_{fand}$, and we denote by $C_{fand}$ the class of $\mathrm{C}^*$-algebras with finite asymptotic nuclear dimension,
\[ C_{fand}  = \cup_{n\geq 0} C_{fand}^{(n)}.\]

The two previous result combines in the main result of \cite{OY4}.
\begin{satz}
Let $A$ be a filtered $\mathrm{C}^*$-algebra with finite asymptotic nuclear dimension. Then $A$ satisfies the Künneth formula. \end{satz}

As an application, H. Oyono-Oyono and G. Yu prove that the uniform Roe algebra of a coarse space with finite asymptotic dimension satisfies the Künneth formula. We have already generalized this result in section \ref{Section:Applications}, but using their ideas we can push even further, as we will see at the end of this section.

One crucial example of controlled Mayer Vietoris pair is given by any decomposition of the base space of an \'etale groupoid with compact base space. Let $G$ be such a groupoid and $U^0$ and $U^1$ two open subsets in $G^{(0)}$ such that
\[G^{(0)} = U^0 \cup U^1. \]
Recall (\cite[Chapter~3]{DellAieraThesis},\cite{dell2017controlled}) that the set $\mathcal E$ of symmetric compact subsets of $G$ is a coarse structure with respect to which $C_r^*(G)$ is filtered by the family of subspaces 
\[C_E(G) = \{ f\in C_c(G) \text{ s.t. supp}(f)\subseteq E  \}\]
indexed by $E\in \mathcal E$.\\

For any open subset $U\subseteq G$, define $U_E$ to be the partial orbit of $U$ by $E$, i.e. $s(int(E)^U)$, and $G^{(E)}_U$ to be the groupoid generated by $G_{|U}\cap E$. Then $U_E$ is an open subset of $G^{(0)}$ and $G^{(E)}_U$ is an open subgroupoid of $G$.\\
 
Given the decomposition $G^{(0)}= U^0 \cap U^1$, set $F_i$ to be the closed subspace $C_0(G_{U^i})\cap C_E(G)$, and $A_i$ to be $C_r^*(G_{U^i_E}^{(E)})$, then 
\[\left( \ F_0 \ , \ F_1 \ , \ A_0 \ , \ A_1 \  \right)\]
is a $E$-controlled Mayer-Vietoris pair for $C_r^*(G)$. \\

Suppose now that a groupoid can be decomposed in such a way at every order into subgroupoids whose reduced $\mathrm{C}^*$-algebra satisfies the Künneth formula. The previous permanence result shows that the reduced $\mathrm{C}^*$-algebra still satisfies the Künneth formula. 

\begin{prop} \label{stability} Let $G$ be an \'etale groupoid such that, for every symmetric compact subset $E\subseteq G$, there exists a decomposition 
\[G^{(0)} = U^0 \cup U^1\]
such that $C^*_r(G_{U^0_E}^{(E)})$, $C^*_r(G_{U^1_E}^{(E)})$  and $C^*_r(G_{U^0_E}^{(E)}) \cap C^*_r(G_{U^1_E}^{(E)})$ satisfy the quantitative Künneth formula, then so does $C^*_r(G)$.
\end{prop}

This leads us to introduce the following notion.

\begin{defi}
Let $\mathcal G$ and $\mathcal H$ be two families of \'etale open subgroupoids of a fixed \'etale groupoid $W$. \\

We say that $\mathcal G$ is $d$-decomposable over $\mathcal F$ if, for every groupoid $G$ in $\mathcal G$, every symmetric compact subset $E\subseteq G$, there exists a covering of $E^{(0)} = s(E)=r(E)$ by $d+1$ open subsets 
\[E^{(0)} = U_0 \cup ... \cup U_d \] such that the groupoids generated by $G_{|U_i} \cap E$ all belong to the class $\mathcal H$.\\

Let $\mathcal C$ be a family of open sugroupoids of $G$. The coarse family generated by $\mathcal C$ is the minimal family of subgroupoids of $G$ containing $\mathcal C$ which is stable by $2$-decomposition.
\end{defi}

The reason why we are keeping some flexibility on the number $d$ is due to the connection of this notion of decomposition to the \textit{dynamical asymptotic dimension} of an \'etale groupoid introduced in \cite{GWY} by E. Guentner, R.Willett and G. Yu. Unravelling the definition, one gets that the dynamical asymptotic dimension of $G$ is less than $d$ iff $\{G\}$ $d$-decomposes over the class of its relatively compact open subgroupoids. We will however only use the decomposition with $d=2$ in this paper. In that case, we can relate spaces of finite decomposition complexity with coarse groupoids being in the coarse family generated by compact \'etale groupoids. FDC was introduced in \cite{GTY}, and studied in detail in \cite{GTY2}. We define here a relative version of FDC for groupoids, starting with any initial class of subgroupoids. A similar notion was defined in \cite{GWY2} for group action: an action of a countable discrete group $\Gamma$ on a compact Hausdorff space $X$ is said to have \textit{finite decomposition complexity} if the action groupoid $X\rtimes \Gamma$ belongs to the class coarsely generated by its relatively compact open subgroupoids in our sense.

\begin{prop}
Let $X$ be a countable discrete metric space with bounded geometry and $G(X)$ its coarse groupoid. Then $G(X)$ belongs to the class coarsely generated by its relatively compact open subgroupoids iff $X$ has finite decomposition complexity in the sense of \cite{GTY}.
\end{prop}

Let $\mathcal Y$ be a countable family of metric spaces. For $R>0$, a $R$-decomposition of a metric space $X$ over $\mathcal Y$ is a decomposition \[X = X_0 \cup X_1\]
where each metric subspace $X_i$ is a $R$-disjoint union of subspaces $X_{ij}\in \mathcal Y$. Denote $G(\mathcal Y)$ the family of ample groupoids $\{G(Y)\}_{Y\in \mathcal Y}$. The proposition is a direct corollary of the following lemma.

\begin{lemma}
$X$ admits a $r$-decomposition over $Y$ iff $\{G(X)\}$ $2$-decomposes over $G(\mathcal Y)$.
\end{lemma}

\begin{proof}
Let $X= X_0\cup X_1$ a $R$-decomposition over a metric family $\mathcal Y$. Let $U_i = \overline X_i$ be the closure in the Stone-\v{C}ech compactification $\beta X$. The source and range maps $r,s : G(X) \rightrightarrows \beta X$ are continuous and open so that 
\[G(X)_{U_i} = r^{-1}(U_i)\cap s^{-1}(U_i) = \overline {r^{-1}(X_i)\cap s^{-1}(X_i)} = \overline {X_i \times X_i}.\]
This entails $G(X)_{U_i} \cap \overline {\Delta_R} = \overline {(X_i \times X_i )\cap \Delta_R } = \overline{\coprod_j X_{ij}\times X_{ij}}$, which is included in $G(X_i)$. Any compact subset of $G(X)$ being contained in some $\Delta_R$, this concludes one way.\\

Given a compact symmetric set $E$ in $G(X)$, let \[G^{(0)}= U_0 \cup U_1\]
be a decomposition of the base space into open subsets, and denote $G_i$ the open subgroupoids of $G$ generated by $G_{|U_i}\cap E$. Fix $i$ and define on $U_i$ the equivalence relation $\sim$ induced by $G_i$: 
\[x\sim y\quad \text{if }\exists g \in G_i \text{ s.t. }s(g) = x \text{ and } r(g)= y.  \]
Let $\{U_{ij}\}_{j \in J_i}$ the equivalence classes, and $X_{ij} = U_{ij}\cap X$. Then:
\begin{itemize}
\item the $X_{ij}$'s cover $X$,
\item for $i$ fixed, $\{X_{ij}\}$ is $E$-separated. Indeed, suppose $j\neq k$ and there exists $(x,y)\in X_{ij}\times X_{ik}\cap \Delta_R$, then $x\sim y$ which is contradicts $j\neq k$. By definition of $G(X)$, there is a number $R>0$ such that $\overline{\Delta_R} \cap E$ is contained in $E$, and then $\{X_{ij}\}$ is $R$-disjoint. 
\end{itemize}  
\end{proof}

This allows to give an example of a coarse space $X$ that does not embed into a Hilbert space but such that $C^*_u (X)$ satisfies the Künneth formula. Indeed, if a group $\Gamma$ is an extension of $H$ by $K$, then the coarse space $|\Gamma |$ finitely decomposes over $\{|H|,|K|\}$. In \cite{ArzhantsevaTesssera} the authors provide an example of a finitely generated group $\Gamma$ which is not coarsely embeddable into a Hilbert space but is a split extension of groups that do. \\

More precisely, the group is build as $\Gamma= \Z / 2\Z  \wr_G (H\times F)$ where:
\begin{itemize}
\item $G$ is the special Gromov monster group, a finitely generated group which contains an expander graph isometrically in its Cayley graph (whence does not admit a coarse embedding into a Hilbert space), a finitely generated group with Haagerup's property but is not coarsely amenable,
\item $H$ is the special Haagerup monster, 
\item $F$ is a finitely generated free group,
\item $G$ is a $H\times F$-set via a surjective morphism $H\times F \twoheadrightarrow G$.
\end{itemize} 
The restricted wreath product is defined as the semi-product of the finitely suppored functions $G \rightarrow \Z / 2\Z$ by the action of $H\times F$ by translation, so that $\Z / 2\Z  \wr_G (H\times F)$ is a split extension of $\Z / 2\Z  \wr_G F$ by $H$, both of which are coarsely embeddable into a Hilbert space.\\

The previous lemma ensures that $G(\Gamma)$ $2$-decomposes over a family of a-T-menable groupoids, whose $\mathrm{C}^*$-algebras as well as their intersection satisfy the Künneth formula. By Proposition \ref{stability}, $C^*_u(|\Gamma|)$ satisfies the Künneth formula. This can be summarize in the following.

\begin{prop}
There exists a metric space that does not coarsely embed into a Hilbert space such that its uniform Roe algebra satisfies the Künneth formula.
\end{prop}

	\ \newline
	{\bf Acknowledgments}. The content of this paper covers some results from the authors' doctoral thesis. The first one would like to thank his supervisor Siegfried Echterhoff for his constant support and valuable advice. The second author expresses his gratitude to his advisor Hervé Oyono-Oyono for his patience and guidance, and to Rufus Willett, who was kind enough to read and comment on the first drafts of this work. 
	
\bibliographystyle{amsalpha}
\bibliography{Literatur,biblio2}

\end{document}